\documentclass[a4paper,11pt]{amsart}

\usepackage[english]{babel}
\usepackage[bitstream-charter]{mathdesign}
\usepackage{amsfonts}
\usepackage{geometry}
\usepackage{amsmath}
\usepackage[colorlinks = true]{hyperref}
\usepackage{amsmath,amsthm}
\usepackage{graphicx}
\usepackage{subfigure}
\usepackage{float}
\usepackage{arydshln}
\usepackage{multirow}
\usepackage{supertabular}
\usepackage{verbatim}

\newcommand{\rank}{\mathrm{rank\,}}		% catalecticant
\newcommand{\cat}{\mathrm{Cat}}		% catalecticant
\newcommand{\HF}{\mathit{h}}				% Hilbert function
\newcommand{\rk}{\mathrm{rk}}			% rank
\newcommand{\reg}{\mathrm{reg}}		% regularity
\newcommand{\van}{\mathrm{van}}		% Vandermonde

\newcommand{\CC}{\mathbb{C}}
\newcommand{\NN}{\mathbb{N}}
\newcommand{\PP}{\mathbb{P}}

\newcommand{\XX}{\mathbb{X}}

\newcommand{\caB}{\mathcal{B}}
\newcommand{\caF}{\mathcal{F}}

\newcommand{\calU}{\mathcal{U}}
\newcommand{\caW}{\mathcal{W}}

\newcommand{\bfx}{\mathbf{x}}
\newcommand{\bfy}{\mathbf{y}}

\newcommand{\ttF}{\mathtt{F}}

\newcommand{\Alessandro}[1]{#1}

\theoremstyle{plain}
\newtheorem{theorem}{Theorem}[section]
\newtheorem{lemma}[theorem]{Lemma}
\newtheorem{proposition}[theorem]{Proposition}
\newtheorem{corollary}[theorem]{Corollary}

\newtheorem{question}[theorem]{Question}

\theoremstyle{remark}
\newtheorem{definition}[theorem]{Definition}
\newtheorem{remark}[theorem]{Remark}
\newtheorem{example}[theorem]{Example}
\newtheorem{notation}[theorem]{Notation}

\title{On minimal decompositions of low rank symmetric tensors}
\author[B. Mourrain]{Bernard Mourrain}
\address[B. Mourrain]{INRIA Sophia Antipolis M\'editerran\'ee (team Aromath), Sophia Antipolis, France}
\email{bernard.mourrain@inria.fr}

\author[A. Oneto]{Alessandro Oneto}
\address[A. Oneto]{Universitat Polit\`ecnica de Catalunya, Barcelona, Spain}
\email{alessandro.oneto@upc.edu}


\DeclareMathVersion{normal2}
\date{\today}
\begin{document}
	\maketitle
 
	\begin{abstract}
		We use an algebraic approach to construct minimal decompositions of symmetric tensors with low rank. \Alessandro{This is done by using Apolarity Theory and by studying minimal sets of reduced points apolar to a given symmetric tensor, namely, whose ideal is contained in the apolar ideal associated to the tensor. In particular, we focus on the structure of the Hilbert function of these ideals of points. We give a procedure which produces a minimal set of points apolar to any symmetric tensor of rank at most $5$. This procedure is also implemented in the algebra software {\it Macaulay2}.}
	\end{abstract}
	
	\section{Introduction}
	{\it Tensors} are multi-dimensional arrays that can be used to encode large data sets. For applications, it is useful to find convenient ways to represent them and, in the last decades, a lot of research has been focused on {\it additive decompositions}. For a more extensive survey on the relations between theoretical and applied aspects of tensor decompositions, we refer to the book of J. M. Landsberg \cite{Lan12}.
	
	\medskip
	In the space of tensors $\CC^{n_1+1}\otimes\ldots\otimes\CC^{n_d+1}$, we call {\it decomposable} or {\it rank-$1$ tensors} the elements of the type $v_1\otimes\ldots\otimes v_d$, where $v_i \in \CC^{n_i+1}$. Given a tensor $T \in \CC^{n_1+1}\otimes\ldots\otimes\CC^{n_d+1}$, we call {\it tensor decomposition} an expression of $T$ as sum of decomposable tensors, i.e.,
	$$
		T = \sum_{i=1}^r v_{i,1}\otimes\ldots\otimes v_{i,d},\quad \text{ where } v_{i,j}\in\CC^{n_j+1},
	$$
	and the smallest length of such a decomposition is called {\it tensor rank} of $T$. We call {\it rank} of $T$ the smallest possible length of such an expression. Note that this definition generalizes the notion of rank of a matrix which may be defined as the smallest number of rank-$1$ matrices needed to write the matrix as their sum.
	
	\medskip
	An important family of tensors is the one of {\it symmetric tensors}, i.e., the tensors invariant under the action of the permutation group on $d$ objects $\mathfrak{S}_d$ on the space of tensors $\CC^{n_1+1}\otimes\ldots\otimes\CC^{n_d+1}$ by permutation of the factors. In this case, we consider additive decompositions as sums of rank-$1$ symmetric tensors. 
	
	If we consider the case $n_1 = \ldots = n_d = n$, symmetric tensors can be naturally identified with homogeneous polynomials of degree $d$ in $n+1$ variables. For example, if $\{x_0,\ldots,x_n\}$ is a basis for $\CC^{n+1}$, a monomial $x_{i_1}\cdots x_{i_d}$ is identified with the symmetric tensor $\sum_{\sigma \in \mathfrak{S}_d} x_{i_{\sigma(1)}}\otimes\ldots\otimes x_{i_{\sigma(d)}}$. Since rank-$1$ symmetric tensors are the ones of the type $v^{\otimes d} = v\otimes \ldots \otimes v$, with the previous identification, they corresponds to $d$th powers of homogeneous polynomials of degree $1$. Therefore, in the case of symmetric tensors, we rephrase the aforementioned problem on additive decomposition as follows.
	
	\medskip
	Let $S = \CC[x_0,\ldots,x_n] = \bigoplus_{d \geq 0} S_d$ be the standard graded ring of polynomials in $n+1$ variables and with complex coefficients. Here, $S_d$ denotes the vector space of degree $d$ homogeneous polynomials, or {\it forms}.
	
	\begin{definition}
		Let $f \in S_d$ be a form of degree $d$. A {\bf Waring decomposition} of $f$ is an expression as
		$$
			f = \ell_1^d + \ldots + \ell_s^d,\text{ where the } \ell_i \text{'s are linear forms.}
		$$
		The minimal length of such a decomposition is called the {\bf Waring rank}, or {\bf rank}, of $f$. We denote it $\rk(f)$.
	\end{definition}
	
	Hence, our general question is the following.
	
	\begin{question}
		Given $f \in S_d$, what is the rank of $f$? Can we provide a minimal Waring decomposition?
	\end{question}
	
	For {\it general} forms of fixed degree and fixed number of variables, the value of the rank is known due to the result of J. Alexander and A. Hirschowitz \cite{AH95}. In the case of specific polynomials, the question is much more difficult. The case of binary forms (two variables) is very classical and due to J. J. Sylvester \cite{Syl51}. In the case of monomials, E. Carlini, M. V. Catalisano and A. V. Geramita gave a very explicit formula just in terms of the exponents of the monomial \cite{CCG12}. In general, several algorithms have been described, but they efficiently work under certain constrains on the given polynomial \cite{BCMT10,BGI11,OO13}.
	
	\medskip
	Our approach to computing Waring decomposition is algebraic. By Apolarity Theory, minimal Waring decompositions of a given polynomial correspond to sets of reduced points in projective space {\it apolar to the polynomial}, i.e., sets of points whose ideal is contained in the so-called {\it apolar ideal} of the polynomial. This theory is explained in details in the book of A. Iarrobino and V. Kanev \cite{IK06}. Under such a correspondence, the coordinates of the points are the coefficients of the linear forms that can be used to provide a Waring decomposition of the polynomial. In particular, the minimal cardinality of such a set of points coincides with the Waring rank of the polynomial. In this paper, we focus on invariants of ideals of sets of points apolar to a given polynomial as their {\it Hilbert function} and their {\it regularity}.
	
	\medskip
	Although this algebraic approach to Waring decompositions is very classical and it basically goes back to the work of J. J. Sylvester on binary forms, we use a new tool which we believe have potential for further investigation. This is the concept of {\it Waring locus} of a polynomial which is defined as the locus of linear forms that may appear in a minimal Waring decomposition \cite{CCO17}. The idea behind this construction is to find a way to decompose a given polynomial by adding one power at the time, namely by taking {\it step-by-step} a linear form in the Waring locus of the polynomial. In particular, this idea has a twofold use. 
	
	If the Waring locus is as big as possible, i.e., it is dense in the space of linear forms, it means that we can actually pick a {\it random} linear form to start our decomposition. This is what happen for form with rank higher than the {\it generic rank}, i.e., the rank of the general form. On the other hand, if the Waring is (contained in) a proper subvariety of the space of linear forms, we have conditions on the coefficients of the linear form we need to start our decomposition. In this case, with some further analysis on algebraic and geometric properties of the Waring locus, we may find a way to reduce the rank of our polynomial. 
	
	\medskip
	By using these ideas, we describe how to find a minimal Waring decomposition of homogeneous polynomials of low rank, for any number of variables and any degree\Alessandro{; see Theorem \ref{thm: main}.} These methods can be extended to forms of higher rank, but, since the cases to study grows very quickly and they might need some {\it ad hoc} argument, we applied them to completely describe all cases up to rank $5$. 
	
	\medskip
	We think it is worth mentioning that our computations left us with an intriguing algebraic question that should be investigated further. We can consider all the minimal sets of points apolar to a given polynomial and we might look at which algebraic and geometric properties they share. As far as we know, the only results in this direction regard: binary forms, where they obviosuly share the same Hilbert function since they are defined by principal ideals with the generator equal to the rank of the binary form; and monomials, where we know that they are complete intersections with the generators in the same degrees \cite{BBT13}.
	
	\subsection*{Structure of the paper.}
	In Section \ref{sec: basic}, we introduced the necessary background and the tools we use in our computations. These include Apolarity Theory (Section \ref{ssec: apolarity}), regularity of ideals of reduced points (e.g., see Theorem \ref{thm: regularity}), essential number of variables (Section \ref{ssec: essential}) and Waring loci (Section \ref{ssec: loci}). In Section \ref{sec: low}, we use these tools to study minimal sets of \Alessandro{points apolar to polynomial of low rank (e.g., see Proposition \ref{prop: rank 4} and Proposition \ref{prop: rank 5}). In Section \ref{sec: main}, we give our main Theorem \ref{thm: main} where we describe a procedure to find a minimal set of points apolar to any polynomial of rank at most $5$.} In Section \ref{sec:M2}, we implement our computations with the algebra software {\it Macaulay2} \cite{M2}. The code of the package {\tt ApolarLowRank} can be found as ancillary material accompanying the arXiv and \Alessandro{the HAL versions} of the paper or on the personal webpage of the second author.
	
	\subsection*{Acknowledgements.}
\Alessandro{	The second author acknowledges a postdoctoral research fellowship at INRIA - Sophia Antipolis M\'editerran\'ee (France) in the team AROMATH, during which this project started. The second author also acknowledges financial support from the Spanish Ministry of
Economy and Competitiveness, through the Mar\'ia de Maeztu Programme for Units of
Excellence in R\&D (MDM-2014-0445).}
			
	\section{Basic definitions and background}\label{sec: basic}
	We start by recalling some basic definitions and construction.
	
	\subsection{Apolarity Theory}\label{ssec: apolarity}
	One of the most important algebraic tools for studying Waring decompositions of homogeneous polynomials is {\it Apolarity Theory}, which relates Waring decompositions of a polynomial $f$ to ideals of reduced points contained in the so-called {\it apolar ideal} of $f$. For more details, we refer to \cite{IK06}.
	
	\smallskip
	Let $T = \CC[y_0,\ldots,y_n] = \bigoplus_{d\geq 0} T_d$ be a standard graded polynomial ring. We define the {\bf apolar action of $T$ over $S$} by identifying the polynomials in $T$ with partial differentials over $S$; namely,
	$$
		\circ : T \times S \longrightarrow S,~~ (G,f) \mapsto G \circ f := G(\partial_0,\ldots,\partial_n)\cdot f.
	$$
	\begin{definition}
		Let $f \in S_d$. We define the {\bf apolar ideal} of $f$ as
		$$
			f^\perp := \{G \in T ~|~ G \circ f = 0\}.
		$$
		We denote by $A_f$ the quotient ring $T/f^\perp$.
	\end{definition}
	\begin{remark}
		An important and useful property of apolar ideals is that, for any $f \in S_d$, the algebra $A_f$ is {\it Artinian Gorenstein} with socle degree $d$. Actually, also the viceversa is true, i.e., any artinian Gorenstein algebra is isomorphic to $A_f$, for some $f$. This characterization is referred as {\it Macaulay's duality} \cite{Mac16}. 
	\end{remark}
	The following lemma is the key of our algebraic approach to Waring decompositions.
	\begin{lemma}[Apolarity Lemma, {\rm \cite[Lemma
            1.15]{IK06}}]\label{lemma: apolarity}
		Let $f \in S_d$. Then, the following are equivalent:
		\begin{enumerate}
			\item $f = c_1\ell_1^d + \ldots + c_s\ell_s^d$, for some $c_i \in \CC\setminus\{0\}$, $\ell_i \in S_1\setminus\{0\}$;
			\item $I_{\XX} \subset f^\perp$, where $I_{\XX}$ is the defining ideal of $s$ reduced points in $\PP^n$.
		\end{enumerate}
		In particular, if $\XX = \{\xi_1,\ldots,\xi_s\} \subset \PP^n$, with $\xi_i = (\xi_{i,0}:\ldots:\xi_{i,n})$, then $\ell_i = \ell_{\xi_i} := \xi_{i,0}x_0+\ldots+\xi_{i,n}x_n \in S_1$.
	\end{lemma}
	
	\begin{definition}
		Given $f \in S_d$, a set of points $\XX$ such that $I_\XX \subset f^\perp$ is said to be {\bf apolar to $f$}.
	\end{definition}
	
	\begin{example}[Binary forms: {\it Sylvester algorithm}]\label{example: sylvester}
		We describe here how to compute the Waring rank of a binary form. The idea behind these computations goes back to J. J. Sylvester {\rm\cite{Syl51}}. For a modern exposition, we refer to {\rm\cite{CS11}}. Let $f \in \CC[x_0,x_1]$. By Macaulay's duality, we know that $f^\perp$ is artinian Gorenstein and, since we are in codimension $2$, it is also a complete intersection, say $f^\perp = (G_1,G_2)$, with $\deg(G_i) = d_i$, $i = 1,2$, and $d_1 + d_2 = d+2$. Since ideals of reduced points in $\PP^1$ are principal, we look for square-free polynomials in $f^\perp$. In particular, we get the following (we assume $d_1 \leq d_2$):
		\begin{enumerate}
			\item if $G_1$ is square-free, then $\rk(f) = d_1$;
			\item otherwise, the general element $H\cdot G_1 + \alpha G_2$, with $H\in T_{d_2-d_1}$, $\alpha \in \CC$, is square-free and $\rk(f) = d_2$.
		\end{enumerate}

        \end{example}

	Another classical tool  useful to analyse these ideals are {\it Hilbert functions}.
	\begin{definition}
		Given a homogeneous ideal $I \subset S$, the {\bf Hilbert function} in degree $i$ of the quotient ring $S/I$ is the dimension of $S_i/I_i$ as $\CC$-vector space, i.e.,
		$$
			\HF_{S/I}(i) := \dim_{\CC}(S/I)_i = \dim_{\CC}S_i - \dim_{\CC} I_i, \text{ for } i \in \NN.
		$$
	\end{definition}
	\begin{remark}\label{rmk: HF points}
		Given a set of reduced points $\XX$, we denote the Hilbert function of the quotient ring $S/I_{\XX}$ simply by $\HF_{\XX}$. A well-known fact is that this Hilbert function is {\it strictly increasing} until it reaches the cardinality of the set of points and then it gets constant \cite[Theorem 1.69]{IK06}.
	\end{remark}
	\begin{remark}
		Since the apolar algebra $A_f$ of a homogeneous polynomial $f \in S_d$ is artinian Gorenstein with socle degree $d$, we know that the Hilbert function of $A_f$ is symmetric and equal to $0$ from degree $d+1$.
	\end{remark}
	Given a polynomial $f \in S_d$, the computation of the apolar ideal is a linear algebra exercise. For any $i = 0,\ldots,d$, we construct the {\bf $i$-th catalecticant matrix} of $f$ as
	$$
		\cat_i(f) : T_i \longrightarrow S_{d-i},~~ G \mapsto G\circ f.
	$$
	Then, we have that, $f^\perp_i = \ker\cat_i(f)$. 
	
\begin{remark}
	For any degree $d$, we consider the standard monomial basis $$\caB_d = \left\{\bfx^\alpha := x_0^{\alpha_0}\cdots x_n^{\alpha_n}~|~ \alpha \in \NN^{n+1}, |\alpha| = \sum_i \alpha_i = d\right\}$$ of $S_d$, and the dual basis
	$$
		\caB^\vee_d = \left\{\bfy^{(\alpha)} := \frac{1}{\alpha!}\bfy^\alpha = \frac{1}{\alpha!}y_0^{\alpha_0}\cdots y_n^{\alpha_n}~|~ \alpha \in \NN^{n+1}, |\alpha| = \sum_i \alpha_i = d\right\}.
	$$
	Note that $\bfy^{(\alpha)} \circ \bfx^\beta = \bfx^{\beta -
          \alpha}$. Therefore, with respect to these basis, we have
        that,
        $$
        \cat_i(f) = (c_{\alpha+\beta})_{\substack{|\alpha| = i \\
            |\beta| = d-i}}, \quad \text{ where } f = \sum_{\alpha \in
          \NN^{n+1},~ |\alpha| = d} c_\alpha \bfx^\alpha \in S_d.
        $$
\end{remark}

By Apolarity Lemma, for any set of points $\XX$
        apolar to $f$, we have that $\HF_{\XX}(i) \geq \HF_{A_{f}}(i)= \mathrm{codim}
        \ker \cat_i(f)$ $= \rk\, \cat_i(f)$. Moreover, for any $i\in \NN$, $|\XX|\geq
        \HF_{\XX}(i)$. 
Therefore, if we denote by $\ell(f) := \max_i\{\HF_{A_f}(i)\}=
\max_i\{\rk \, \cat_{i}(f)\}$ the {\it differential length} of $f$, we have that
	$$
		\rk(f) \geq \ell(f).
        $$
\begin{remark}
If $f$ is a binary form, then $f^{\perp}=(G_{1},G_{2})$ with
$\deg(G_{1})=d_{1}, \deg(G_{2})=d_{2}$, $d_{1}\le d_{2}$ and $d_{1}+d_{2}=d+2$.
Then $h_{f}(i)=i+1$ for $i=0,\ldots,d_{1}-1$, $h_{f}(i)=d_{1}$ for
$i=d_{1},\ldots,d_{2}-1$ and $h_{f}(i)=d+1-i$ for $i=d_{2},\ldots, d$.
In particular, $l(f)=d_{1}$ and $\rk(f)=l(f)$ if $G_{1}$ is
square-free. Otherwise, $\rk(f)= d+2-l(f)$.
\end{remark}          

\begin{lemma}\label{lemma:fk}
If for some $k\le d$, $(f_{k}^{\perp})=I_{\XX}$ is defining a set
$\XX$ of $r$
reduced points, then $\rk(f)\le r=h_{f}(k)$.
\end{lemma}
\begin{proof}
As $f_{k}^{\perp}$ is defining a set $\XX$  of $r$
reduced points, $I_{\XX}= (f_{k}^{\perp})\subset (f^{\perp})$ and by
the apolarity Lemma \ref{lemma: apolarity}, $\XX$ is apolar to $f$ and
$\rk(f)\le |\XX|=r$.
Moreover, $h_{\XX}(k)=r=\dim(S_{k}/(I_{\XX})_{k})=\dim(S_{k}/f_{k}^{\perp})=h_{f}(k)$.
\end{proof}
 
This leads to the following possible algorithm to find the Waring rank of a given polynomial $f \in S_d$:
	\begin{enumerate}
		\item consider the largest catalecticant $\cat_m(f)$, for $m = \left\lfloor \frac{d}{2} \right\rfloor$ and the ideal $I$ generated by its kernel;
		\item if $I$ does not define a set of reduced points, then we fail;
		\item otherwise, if the zero set of $I$ is a set of reduced points $Z(I) = \{[L_1],\ldots,[L_r]\}$, then we solve the linear system $f = \sum_{i=1}^r c_i \ell_i^{d}$ to find a Waring decomposition of $f$. Moreover, in this case, this is minimal and unique.
	\end{enumerate}
Numerical conditions to ensure that this catalecticant method works have
been presented in \cite{IK06,OO13}. 

\medskip
In \cite{IK06}, A. Iarrobino and V. Kanev analysed the Hilbert function of ideals of sets of reduced points apolar to a given polynomial in order to use Apolarity Lemma and deduce its rank. We want to continue in this direction and, in the next section, we will classify polynomials with low rank.

\begin{definition}[Regularity]
  For a family $\XX=\{\xi_{1},\ldots,\xi_{r}\}$ of points in $\PP^{n}$, we define the {\bf regularity} of $\XX$ as
$$
\rho(\XX)= \min \{k\in \NN \mid \exists U_{1},\ldots,U_{r} \in S_{k}\ \mathrm{s.t.}\ U_{i}(\xi_{j})=
\delta_{i,j}\}.
$$
\end{definition}
\begin{remark}This regularity is also called the {\em interpolation degree} of the points $\XX$.
Let $\van_k(\XX)$ denotes the Vandermonde matrix of degree $k$ associated to $\XX$, i.e., if $\XX = \{\xi_1,\ldots,\xi_r\}$, with $\xi_i = (\xi_{i,0}:\ldots:\xi_{i,n}) \in \PP^n$, 
$$
\van_{k}(\XX) = \left( \xi_{j}^\alpha \right)_{\substack{j = 1,\ldots,r \\ |\alpha| = k}}, 
$$
where $\xi_j^\alpha := \xi_{j,0}^{\alpha_0}\cdots \xi_{j,n}^{\alpha_n}$.
The regularity $\rho(\XX)$ is also the minimal $k$ for which,
$\van(\XX)_{k}$ is of rank $|\XX|$. 

This regularity coincides with the so-called {\it regularity index},
i.e., the smallest integer in which the Hilbert function of the ideal
of points gets constant. Also, $\rho(\XX)$ is the {\it
  Castelnuovo-Mumford regularity} of $S/I_\XX$ which is defined as
$\min_{i}\{d_{i,j}-i\}$ where $d_{i,j}$'s are the degrees of
generators of the $i$-th syzygy module in a minimal free resolution of
$S/I_\XX$; see \cite[Chapter 4]{eisenbud_geometry_2005}: $\rho(\XX)=
\mathrm{reg}(S/I_{\XX})= \mathrm{reg}(I_{\XX})-1$.
\end{remark}

\begin{question}\label{question:A}
	Let $\XX,\XX'$ be minimal set of points apolar to a polynomial $f \in S_d$.
	Is it true that $\rho(\XX) = \rho(\XX')$? 
	
	More generally, is it true that $\HF_\XX = \HF_{\XX'}$?
\end{question}

The latter question has a positive answer for:
\begin{enumerate}
	\item binary forms, as described by Sylvester's algorithm;
	\item monomials, since any minimal apolar set of points to a monomial $x_0^{d_0}\cdots x_n^{d_n}$, where the exponents are ordered increasingly, is a complete intersection with $n$ generators of degrees $d_1+1,\ldots,d_n+1$, respectively; see \cite{BBT13},
\end{enumerate} 

We now prove that it has an affirmative answer also if the regularity of a minimal set of points is large enough with respect to the degree of the polynomial. In particular, in this case, we have that the catalecticant method works and gives us a minimal apolar set of points.

\begin{lemma}\label{lemma: regularity}
  Let $f \in S_d$ and let $\XX$ be a minimal set of points apolar to $f$. Assume that $d \geq \rho(\XX)$.
  Then,
  $$
  (I_{\XX})_{k} = f^\perp_{k}\quad
  \mathrm{for}\ 0\le k\le d-\rho(\XX).
  $$
\end{lemma}
\begin{proof}
Let $\XX = \{\xi_1,\ldots,\xi_r\}$, where $\xi_i = (\xi_{i,0}:\ldots:\xi_{i,n}) \in \PP^n$. We denote by $\ell_{\xi_i} = \xi_0x_0 + \ldots + \xi_nx_n \in S_1$. By Apolarity Lemma, we know that $f = \sum_{i=1}^s a_i \ell_{\xi_i}^d$, for some coefficients $a_i \in \CC$. Now, for any $\bfy^{(\alpha)} \in T_k$, we have that 
\begin{align*}
\bfy^{(\alpha)} \circ f & = \sum_{i=1}^s a_i \bfy^{(\alpha)} \circ \ell_{\xi_i}^d = 
\sum_{i=1}^s \left(a_i  \frac{d!}{(d-k)!}\right) \xi_i^\alpha\ell_{\xi_i}^{d-k} = \\
& = \sum_{\substack{\beta \in \NN^{n+1} \\ |\beta| = d-k}}\sum_{i=1}^s \left( a_i\frac{d!}{\beta_0!\cdots \beta_n!}\right) \xi_i^{\alpha+\beta} \bfx^\beta =  \sum_{\substack{\beta \in \NN^{n+1} \\ |\beta| = d-k}}\sum_{i=1}^s \overline{a}_i\xi_i^{\alpha+\beta} \bfx^\beta.
\end{align*}
Therefore, 
$$
	\cat_i(f) = \left(\sum_{i=1}^s \overline{a}_i\xi_i^{\alpha+\beta}\right)_{\substack{|\alpha| = k \\ |\beta| = d-k}} = \van_{d-k}(\XX)^T \cdot D \cdot \van_k(\XX),
$$
where $D$ is the diagonal matrix $D = {\rm diag}(\overline{a}_1,\ldots,\overline{a}_s)$.

        Since $d-k\ge \rho$, $\van_{d-k}(\XX)^{T}$ is injective. Therefore, we have that the kernel of $\cat_{k}(f)$, which is $f^\perp_{k}$, is equal to the kernel of $\van_{k}(\XX)$, which is $(I_{\XX})_{k}$. 
\end{proof}

\begin{remark}
	In \cite[Theorem 5.3(E-ii)]{IK06}, the authors proved a similar statement under a stronger assumption, namely, by assuming that the polynomial admits a {\it tight} apolar set of points, i.e., a set of points $\XX$ apolar to $f$ such that $\HF_{A_f}(i) = |\XX|$, in some degree $i$. 
\end{remark}
\begin{theorem}\label{thm: regularity}
  Let $f \in S_d$ and let $\XX$ be a minimal set of points apolar to $f$. If
  $d \geq 2\rho(\XX)+1$, then $I_{\XX} = (f^\perp_{\leq\rho(\XX)+1})$. Moreover, $\XX$ is the unique minimal set of points apolar to $f$.
\end{theorem}
\begin{proof}
By Lemma \ref{lemma: regularity}, for $0\le k\le \rho(\XX)+1$, we have $(I_{\XX})_{k} = f^\perp_{k}$.
Since $\rho(\XX)+1=\reg(I_{\XX})$ is greater than the degree of a
minimal set of generators of $I_{\XX}$, $(f^\perp_{\leq\rho(\XX)+1})= I_{\XX}$.
\end{proof}

\subsection{Essential number of variables}\label{ssec: essential}
	In \cite{Car06}, E.~Carlini introduced the concept of {\it essential number of variables} of a polynomial as the smallest number of variables needed to write it.
	
	\begin{definition}
		Given a homogeneous polynomial $f \in S$, the {\bf essential number of variables} of $f$ is the smallest number $N$ such that there exists linear forms $\ell_1,\ldots,\ell_{N} \in S$, such that $f \in \CC[\ell_1,\ldots,\ell_{N}]$. 
		In this case, we call the $\ell_i$'s the {\bf essential variables} of $f$. 	In the literature, a form $f \in \CC[x_0,\ldots,x_n]$ with $n+1$ essential variables is also called {\bf concise}.
	\end{definition}
	
	\begin{lemma}\label{lemma: essential var}
		Let $f \in S_d$. Then:
		\begin{enumerate}
			\item {\rm \cite[Proposition 1]{Car06}} the
                          number of essential variables is the rank of
                          $\cat_1(f)$, that is $h_{f}(1)$;
			\item {\rm \cite[Proposition 2.3]{CCO17}} any minimal Waring decomposition of $f$ involves only linear forms in the essential variables.
		\end{enumerate}
	\end{lemma}
	For this reason, the first thing we do when we look for a Waring decomposition is to compute the first catalecticant matrix and then working modulo its kernel.
	
	\begin{example}[Rank $1$ polynomials]\label{example: rank 1}
		If $f$ has only one essential variable, i.e., the first catalecticant matrix has rank $1$, then we have that $f$ is a pure $d$-th power of a linear form. Indeed, if we consider the kernel of the first catalecticant matrix we obtain $n$ linear forms which define a simple points $\xi \in \PP^n$. Then, by Apolarity Lemma, for a suitable choice of a scalar $c \in \CC$, $f = c\ell_\xi^d$.
	\end{example}
	
	\subsection{Waring loci and forms of high rank.}\label{ssec: loci}
	In \cite{CCO17}, the second author together with E. Carlini and M.V. Catalisano defined the concept of {\it Waring locus} of a homogeneous polynomial.
	
	\begin{notation}
		Given a subset $W$ of elements in a vector space, we denote by $\langle W \rangle$ their linear span. Similarly, if we consider a subset of points in a projective space, it will denote their projective linear span.
	\end{notation}
	
	\begin{definition}
		Let $f \in S_d$. Then, the {\bf Waring locus} of $f$ is the locus of linear forms that can appear in a minimal Waring decomposition of $f$, i.e.,
		$$
			\caW_f := \left\{ [\ell] \in \PP(S_1) ~|~ \exists \ell_2,\ldots,\ell_r,~r = \rk(f),~\text{s.t.}~f \in \left\langle \ell^d, \ell_2^d, \ldots, \ell_r^d \right\rangle\right\};
		$$
		analogously, by Apolarity Lemma,
		$$
			\caW_f := \left\{ P \in \PP^n ~|~ \exists P_2,\ldots,P_r,~r = \rk(f),~\text{s.t.}~I_{\XX} \subset f^\perp,~\XX = \{P,P_2,\ldots,P_r\}\right\}.
		$$
		The complement is called {\bf forbidden locus} of $f$ and denote $\caF_f := \PP^n \setminus \caW_f$.
	\end{definition}
	\begin{remark}
		The Waring locus (hence, the forbidden locus) is not necessary open or closed, e.g., in the case of planar cubic cusps it is given by the union of a point and a Zariski open subset of a line; see \cite[Theorem 5.1]{CCO17}. We only know that it is constructible since it can be described as a linear projection of (the open part of) the classical {\it Variety of Sums of Powers} (VSP) defined by K. Ranestad and F.-O. Schreyer \cite{RS00}, i.e., $
			{\it VSP}(f,\rk(f)) := \overline{\left\{ ([\ell_1],\ldots,[\ell_r]) \in {\it Hilb}_s(\PP(S_1)) ~|~ f \in \left\langle\ell_1^d,\ldots,\ell_r^d \right\rangle\right\}}.
		$
	\end{remark}
	The motivation that inspired the definition of Waring loci is to look for a {\it recursive} way to construct Waring decompositions, by adding, step-by-step, one power at the time. In \cite{CCO17}, Waring loci of quadrics, binary forms, monomials, plane cubics have been computed. 
	
	\begin{example}[Recursive decomposition of binary forms]
		In {\rm\cite[Theorem 3.5]{CCO17}}, the Waring locus of binary forms has been computed. By Sylvester's algorithm, if the rank $r = \rk(f)$ is less than the generic, i.e., $r < \left\lceil \frac{d+1}{2} \right\rceil$, or $r = \left\lceil \frac{d+1}{2} \right\rceil$ and $d$ is odd, then, we have a unique decomposition and, in particular, the Waring locus is closed and consists of $r$ distinct points. If $r > \left\lceil \frac{d+1}{2} \right\rceil$ or $r = \left\lceil \frac{d+1}{2} \right\rceil$ and $d$ is even, then, the Waring locus is dense. This means that, in the latter cases, for a general form $\ell\in S_1$, there exists a minimal Waring decomposition of $f$ involving $\ell^d$, up to some scalar. Actually, by {\rm \cite[Proposition 3.8]{CCO17}}, we know that for a general choice of $\ell_1,\ldots,\ell_s \in S_1$, where $s = r - \left\lceil \frac{d+1}{2} \right\rceil$, there exist scalars $c_1,\ldots,c_s$ such that $f - \sum_{i=1}^s c_i\ell_i^d$ has rank $r - s$. At this point, we cannot continue with generic linear forms because, depending on the parity of the degree, the remaining part of the decomposition might be uniquely determined.
	\end{example}
	
Our first result is a generalization of the fact explained in the latter example in a more general setting.

\begin{definition}
For any projective variety $X\subset \PP^N$,
we say that $X$ {\em spans} $\PP^{N}$ if every point of $\PP^{N}$ is in  the linear span of points in $X$. 

Given a point $P \in \PP^N$, the {\bf $X$-rank} of $P$ is the smallest number of points on $X$ whose linear span contains $P$. We denote it
$\rk_X(P)$. By convention, if $P$ is not in any linear span of points
of $X$, $\rk_X(P)=+\infty$.
\end{definition}
	
\begin{remark}
From this definition, the Waring rank is simply the $X$-rank inside the space of homogeneous polynomials of $\PP(S_d)$ with respect to the {\it Veronese variety} of $d$-th powers. Other relevant varieties that have been considered in relation to {\it tensor decompositions} are {\it Segre} and {\it Segre-Veronese varieties}. 
\end{remark}
	
\begin{definition}
Given a point $P\in \PP^N$, we define the {\bf $X$-decomposition locus} of $P$ as
	$$
		\caW_{X,P} = \left\{ Q\in X ~|~ \exists
                Q_2\in X,\ldots,Q_r\in X,~r = \rk_X(P),~~ P \in \left\langle Q,Q_1,\ldots,Q_r\right\rangle\right\}.
	$$
	The {\bf $X$-forbidden locus} is $\caF_{X,P} = X \setminus \caW_{X,P}$.
\end{definition}

\begin{remark}
	If $X$ is the Veronese variety of $d$-th powers of linear forms, the $X$-decomposition locus of a poin $[f] \in \PP(S_d)$ corresponds to the image of the Waring locus of $f$ via the $d$-th Veronese embedding. Analogously for the forbidden locus.
\end{remark}

In the following, we prove that the $X$-decomposition locus of a point with rank higher than the generic is dense in $X$. The proof follows an idea used in \cite{BHMT17} to study the loci of points with high rank.

\begin{theorem}\label{thm: Waring loci supgeneric}
Let $X\subset \PP^N$ be an irreducible projective variety, which spans $\PP^{N}$ and let $g$ be the generic $X$-rank. Let $P\in\PP^N$ with $r = \rk_X(P)$. If $r > g$, then $\caW_{X,P}$ is dense in $X$.
\end{theorem}
\begin{proof}
We proceed by induction on $r$. Assume that $P$ has $X$-rank $r >
g+1$. Then it lies on a line $\langle P_1, P_2\rangle$, where $P_1$
has $X$-rank $g+1$ and $P_2$ has $X$-rank $r-g-1$. Now, if we assume
that the claim holds for $P_1$, we have that, for a general point $Q
\in X$, we have a point $Q' \in \langle P_1,Q\rangle$ of $X$-rank
$g$. Now, let $P'$ be the point of intersection $\langle P, Q\rangle
\cap  \langle Q' , P_2\rangle$. Since $P' \in \langle Q', P_2\rangle$,
then $\rk_X(P') = r-1$. Hence, $P\in \langle P', Q\rangle$ with $P'$
of $X$-rank $r-1$ and $Q\in X$ so that $Q \in \caW_{X,P}$.
	
Hence, we just need to prove the claim in the case $\rk_X(P) =
g+1$. Let $\sigma^\circ_g$ be the set of points of $X$-rank equal to
$g$. By definition of the generic rank, we know that $\sigma^\circ_g$
is a dense subset of $\PP^N$. For any $P \in \PP^{N}$ of rank $g+1$,
let $C_P = \langle P, X\rangle$ be the union of all lines passing
through $P$ and a point on $X$. As $P$ is of $X$-rank $g+1$, it is on
a line $\langle P', Q\rangle$ with $P' \in \sigma^\circ_g$ of $X$-rank
$g$ and $Q\in X$. Thus $C_{P} \cap \sigma^\circ_g$ is non-empty. As
$X$ and $C_{P}$ are irreducible and $\sigma^\circ_g$ is dense in
$\PP^{N}$, the Zariski closure of $C_{P} \cap \sigma^\circ_g$ is
$C_{P}$ and $C_P \cap \sigma^\circ_g$ is dense in $C_P$. 
Therefore, for a generic point $Q \in X$,
there is a point $P'\in \sigma^\circ_g$ with $X$-rank equal to
$g$ on the line $\langle P, Q\rangle$. By definition of $\caW_{X,P}$,
it implies that $Q \in \caW_{X,P}$. This concludes the proof.
\end{proof}

\begin{corollary}
Let $g$ be the generic rank of forms of degree $d$ in $n+1$ variables. Let $f\in S_d$ with $r = \rk(f)$. If $r > g$, then for any general choice of $\ell_1,\dots,\ell_s \in S_1$, with $s = r - g$, there exists a minimal Waring decomposition involving the $\ell_i$'s.
\end{corollary}
\begin{proof}
It directly follows by applying $r-g$ times Theorem \ref{thm: Waring loci supgeneric} on Veronese varieties.
\end{proof}

\begin{remark}
	A big challenge when we want to use Waring loci to construct minimal Waring decompositions is that, fixed a linear form $\ell$ in the Waring locus of $f$, {\it there exists} a suitable coefficient such that $\rk(f+c\ell^d) = \rk(f)-1$, but computing the scalar $c$ is not trivial. In the case of forms of high rank, we have seen that $\ell$ can be chosen generically, but this also implies that the scalar $c$ can be chosen generically. 
	
	Indeed, let $f$ be of rank $r$ higher than the generic and let $\ell$ be a general linear form. Since the (closure of the) locus of forms of rank bigger than $r$ is a proper subvariety of forms of rank $r-1$, we have that on the line $\langle f, \ell^d\rangle$ the condition of having rank $r-1$ is an open condition. Therefore, since it is also non empty because $\ell$ is in the Waring locus of $f$, the general point of the line has rank $r-1$.
\end{remark}

\begin{remark}
In the proof of Theorem \ref{thm: Waring loci supgeneric}, the fact that the point has rank strictly larger than the generic rank is crucial. Indeed, if we consider forms of smaller rank, anything can happen. For example:
\begin{enumerate}
\item The Waring locus can be {\it open}: if we consider a general
  plane cubic of rank $4$, we know that the apolar ideal is generated
  by three conics which define a linear base-point-free linear system;
  hence, by Bertini's Theorem, if we impose the passage through a
  point $P$ of the plane, we obtain a pencil of conics. These conics
  define a set of four reduced points, for $P$ outside a discriminant
  curve of $\PP^{2}$. The Waring locus is the complementary of this
  discriminant curve and is open (see for more details \cite[Section 3]{CCO17}); 
\item The Waring locus can be {\it closed} and {\it $0$-dimensional}: a general cubic of $\PP^3$ has rank $5$ and, by Sylvester Pentahedral Theorem \cite{Cle61}, we have that it is {\it identifiable}, i.e., has a unique decomposition. Therefore, the Waring locus is the unique minimal apolar set of points. It is classically known that also the general binary form of odd degree and the general plane quintic are identifiable. Recently, Galuppi and Mella proved that this are the only cases \cite{GM16}. An algorithm to find such a unique decomposition is presented in \cite[Theorem 3.9]{OO13}.
	\item The Waring locus can be {\it neither closed nor open}: consider a plane cuspidal cubic which has rank $4$ and, up to a change of variables, it can be written in the form $x_0^3 + x_1^2x_2$. The Waring locus is given by the union of the point $(1:0:0)$ and the pinched line $\PP^1_{x_1,x_2}\setminus (0:1)$ (see \cite[Section 3]{CCO17}).
\end{enumerate}
\end{remark}

In the following lemma, we generalize the latter case to a more general setting. 

\begin{lemma}\label{lemma: general cusps}
		Let $f = x_0^d + g(x_1,\ldots,x_n) \in S_d$ of rank $n+2$, with $n \geq 2$, and degree $d \geq 4$. Assume that $g^\perp_2 = [(y_0,G_1,\ldots,G_N)]_2$, where the $G_i$'s are quadrics and $N = {n \choose 2} - 1$. Then, $\caW_f = (1:0:\ldots:0) \cup \caW_g$. In other words, any minimal Waring decomposition of $f$ is given by $x_0^d$ plus a minimal decomposition of $g$.
	\end{lemma}
	\begin{proof}
		By \cite[Lemma 1.12]{BBKT15}, we know that $f^\perp_i = ((y_0^{d})^\perp)_i \cap g^\perp_i$, for $i \leq d-1$. In particular, since $d \geq 4$, we get $f^\perp_2 = (y_0y_1,\ldots,y_0y_n,G_1,\ldots,G_N)$. Hence, we have that 
		$$
			\HF_{A_f}(2) = {n+2 \choose 2} - n - \left({n \choose 2} - 1\right) = n+2.
		$$
		Therefore, since any minimal set $\XX$ of points apolar to $f$ has rank $n+2$, we obtain $(I_{\XX})_2 = f^\perp_2$. Hence, $\XX$ is contained in the variety defined by $(f^\perp_2)$ which is 
		$(1:0:\ldots:0) \cup \PP^{n-1}_{x_1,\ldots,x_n}$. Hence, any minimal decomposition is of the type $x_0^d + \sum_{i=1}^r \ell_i^d(x_1,\ldots,x_n)$, where $r = \rk(g)$. By restricting on $\{x_0 = 0\}$, we get a minimal decomposition of $g$ and the claim follows. 
	\end{proof}
	
\section{Decompositions of low rank polynomials}\label{sec: low}
From the previous section, we noticed that if the degree of
the polynomial is sufficiently large with respect to the regularity of
the points of a minimal decomposition, then there is a unique Waring decomposition which can be found directly from the generators of the apolar ideal (Theorem \ref{thm: regularity}). Also, we noticed that if the rank is sufficiently large, then we can choose some elements of a minimal Waring decomposition generically and reduce the rank to be equal to the general rank (Theorem \ref{thm: Waring loci supgeneric}).

	In this section, we use these tools to construct minimal Waring decompositions of polynomials of small rank, for {\it any} number of variables and {\it any} degree. 
	
\begin{remark}
	Any quadric $q(\bfx)$ can be represented by a symmetric matrix $Q$, i.e., $q(\bfx) = \bfx Q \bfx^T$. Then, it is well known that the Waring rank of $q$ coincides with the rank of $Q$ and a minimal Waring decomposition is obtained by finding a diagonal form of $Q$. Therefore, we will always assume $d \geq 3$.
\end{remark}

\begin{example}[Rank $1$ and $2$]\label{example: low rank}
	The rank $1$ case can be easily explained in terms of essential variables, see Example \ref{example: rank 1}. The rank $2$ case can be explained using the Sylvester algorithm, see Example \ref{example: sylvester}. 
	In particular, if the Hilbert function of $A_f$ is $1 ~~ 2 ~~ 2 ~~ \cdots ~~ 2 ~~ 1 ~~ -$, then it means that $f$ has two essential variables and the apolar ideal is given by
	$$
		f^\perp = (L_1,\ldots,L_{n-1},G_1,G_2), \text{ where }\deg(L_i) = 1, \deg(G_1) = 2, \deg(G_2) = d,
	$$
	where $G_1$ is square-free.
\end{example}

	\begin{lemma}\label{lemma: rank = essential}
		Let $f \in \CC[x_0,\ldots,x_n]$ be a concise form of degree $d$. Then, $\rk(f) = n+1$ if and only if $\ell(f) = n+1$ and $f^\perp_2$ defines a set of reduced points. In this case, there is a unique minimal apolar set of points.
	\end{lemma}
	\begin{proof}
		If $f$ has $n+1$ essential variables and the rank is equal to $n+1$, then, up to a change of coordinate, we can write it as $f = x_0^d + \ldots + x_n^d$. In this case, we know that
		$$
			\HF_{A_f} : ~~1 ~~~ n+1 ~~~ n+1 ~~~ \ldots ~~~ n+1 ~~~ 1 ~~~ - .
		$$
		Hence, any minimal set of points is such that $(I_{\XX})_{2} = (f^\perp_2) = \left\langle y_iy_j ~|~ i,j = 0,\ldots,n\right\rangle$. These quadrics define a set of $n+1$ reduced coordinate points; therefore, also uniqueness follows.
		
		Viceversa, if $I_{\XX} = (f_2^\perp)$ is a set of reduced points and $\ell(f) = n+1$, since $f$ is concise, we have that
		$$
			\HF_{S/I_{\XX}} : ~~ 1 ~~~ n+1 ~~~ n+1 ~~~ n+1 ~~~ \cdots .
		$$
		Therefore, $|\XX| = n+1$ and, by Apolarity Lemma, we have $\rk(f) = n+1$.
	\end{proof}
	
	Now, we can start our analysis of minimal decompositions of low rank polynomials.
	
\subsection{Polynomials of rank $3$.}
If the rank of $f$ is equal to $3$, then we have that $f$ has at most three essential variables. Hence, we only have two possible configurations of points:
\begin{enumerate}
	\item[(3a)] three collinear points;
	\item[(3b)] three general points. 
\end{enumerate}
\begin{proposition}\label{prop: rank 3}
		Let $f\in S_d$ be a form of rank $3$.
		\begin{enumerate}
			\item[{\rm (3a)}] If $f$ has two essential variables, then $f^\perp = (L_1,\ldots,L_{n-1},G_1,G_2)$, where we set $d_i := \deg(G_i)$ and $d_1 \leq d_2$. In particular, 
			\begin{enumerate}
				\item[(i)] for $d = 3,4$, minimal apolar sets of points are given by 
				$I_\XX = (L_1,\ldots,L_{n-1}, HG_1 + \alpha G_2)$, for a general choice of $H \in T_{d_2-d_1}$ and $\alpha \in \CC$;
				\item[(ii)] if $d \geq 5$, there is a unique minimal apolar set of points given by $I_\XX = (L_1,\ldots,L_{n-1},G_1)$.
			\end{enumerate}
			\item[{\rm (3b)}] If $f$ has three essential variables, there is a unique minimal apolar set of points given by $I_{\XX} = (f^\perp_2)$.
		\end{enumerate}
	\end{proposition}
	\begin{proof}
	Case (3a) follows from Sylvester algorithm (Example \ref{example: sylvester}) and case (3b) from Lemma \ref{lemma: rank = essential}.
	\end{proof}
	\begin{remark}
		By Lemma \ref{lemma: essential var}, we obtain a stratification of the locus of rank $3$ polynomials in the sense that, given a polynomial of rank $3$, {\it all} minimal apolar sets of points are either of type (3a) or of type (3b). In this way, we have that Question \ref{question:A} has positive answer for rank $3$ polynomials. In particular, we obtain that: if $f$ is of type (3a), then any minimal apolar set of points have Hilbert function $1 ~~ 2 ~~ 3 ~~ 3 ~~ \cdots$; while, if $f$ is of type (3b) then any minimal apolar set of points have Hilbert function $1 ~~ 3 ~~ 3 ~~ 3 ~~ \cdots$.
	\end{remark}

\begin{remark}\label{remark: Aronhold invariant}
The Zariski closure of the space of plane cubics of rank $3$ is an
hypersurface defined by the {\it Aronhold invariant}; e.g. see
\cite{LO13}. In \cite{Ott09}, G. Ottaviani describes how to compute
such invariant in terms of Pfaffians of particular skew-symmetric
matrices called {\it Koszul flattenings}. We refer also to \cite{OO13}
for a description of Koszul flattenings of homogeneous polynomials and
their use to compute decompositions of symmetric tensors. 
\end{remark}
 
	\subsection{Polynomials of rank $4$.} 
	The possible configurations of $4$ points in projective spaces are in Figure \ref{fig: 4 points}.
	\begin{figure}[H]
	\begin{center}
		\subfigure[(4a) Collinear]{
			\includegraphics[scale=0.3]{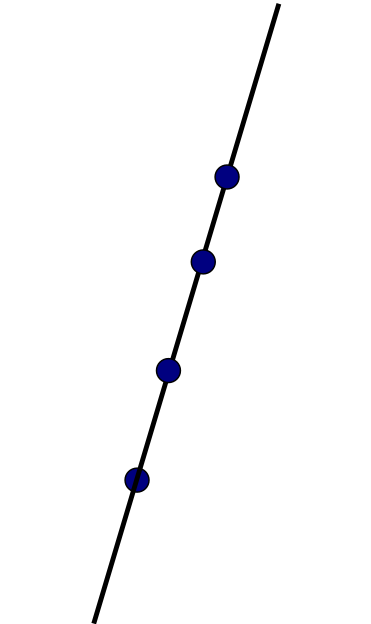}
		}
				\subfigure[(4b) Coplanar, with $3$ collinear.]{
			\includegraphics[scale=0.28]{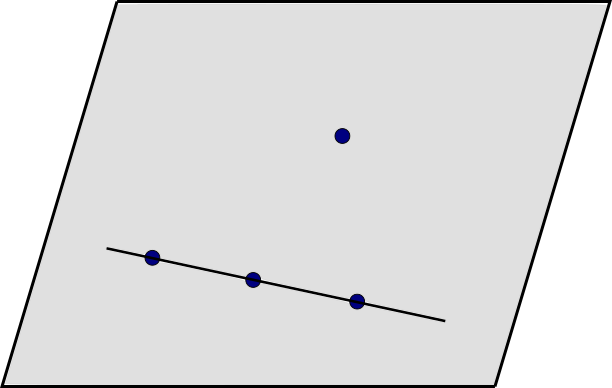}
		}
				\subfigure[(4c) General coplanar.]{
			\includegraphics[scale=0.28]{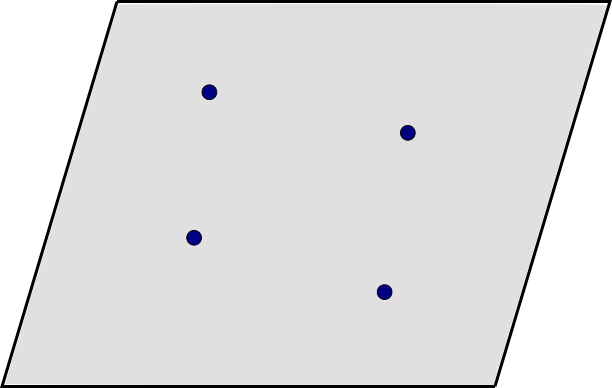}
		}	
				\subfigure[(4d) General points.]{
			\includegraphics[scale=0.28]{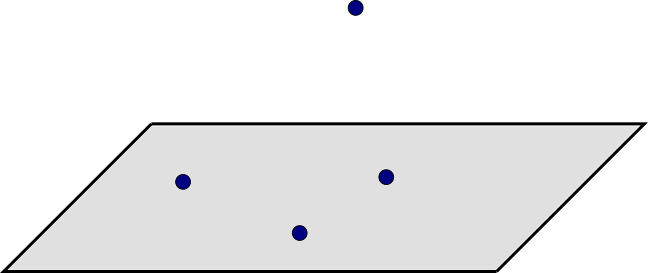}
		}	
		\end{center}
		\caption{Configurations of $4$ points in projective space.}\label{fig: 4 points}
	\end{figure}
	
	\begin{theorem}\label{prop: rank 4}
		Let $f \in S_d$ be a form of rank $4$.
		\begin{enumerate}
		\item[{\rm (4a)}] If $f$ has two essential variables ($h_{f}(1)=2$),
                  then $f^\perp = (L_1,\ldots,L_{n-1},G_1,G_2)$, where $\deg(G_i) = d_i$ and $d_1 \leq d_2$. In particular, it has to be $d \geq 4$ and:
			\begin{enumerate}
				\item[(i)] if $d = 4,5,6$, then $d_2 = 4$, and minimal apolar sets of points are defined by ideals
				$I_\XX = (L_1,\ldots,L_{n-1}, HG_1 + \alpha G_2)$, for a general choice of $H \in T_{6-d}$ and $\alpha \in \CC$;
				\item[(ii)] if $d \geq 7$, then $d_1 = 4$ and the unique minimal apolar set of points  is given by $I_\XX = (L_1,\ldots,L_{n-1},G_1)$.
			\end{enumerate}
			\item[{\rm (4b)}] If $f$ has three essential
                          variables ($h_{f}(1)=3$) and a minimal apolar set $\XX$ of type {\rm (4b)}, then:
			\begin{enumerate}
				\item[(i)] if $d = 3$, then
				$(f^\perp_2)$ defines a $0$-dimensional scheme $P + D$, where $P$ is a reduced point and $D$ is connected scheme of length $2$ whose linear span is a line $L_D$; moreover, any minimal apolar set is of the type $P \cup \XX'$, with $\XX' \subset L_D$;
				\item[(ii)] if $d = 4$, then
                                  $h_{f}(2)=4$, $(f^\perp_2)$ defines a disjoint union $P \cup L$, where $P$ is a reduced point and $L$ is a line not passing through $P$; moreover, any minimal apolar set is of the type $P \cup \XX'$, where $\XX' \subset L$.
				\item[(iii)] if $d \geq 5$, then
                                  $h_{f}(2)=4$, $(f^\perp_2)$ defines a disjoint union $P \cup L$, where $P$ is a reduced point and $L$ is a line not passing through $P$ and $(f^\perp_{3})$ defines the unique minimal apolar set.
			\end{enumerate}
			\item[{\rm (4c)}] If $f$ has three essential
                          variables ($h_{f}(1)=3$) and a minimal apolar set $\XX$ of type {\rm (4c)} then:
			\begin{enumerate}
				\item[(i)] if $d = 3$, then
                                  $Z(f_{2}^{\perp})= \emptyset$ and
                                  $\caW_f$ is dense in the plane of
                                  essential variables;
				\item[(ii)] if $d \geq 4$, there is a unique minimal apolar set of points given by $I_{\XX} = (f^\perp_2)$.
			\end{enumerate}
			\item[{\rm (4d)}] If $f$ has four essential variables, there is a unique minimal apolar set of points given by $I_\XX = (f^\perp_2)$.
		\end{enumerate}
	\end{theorem}
	\begin{proof}
		{\it Case {\rm (4a)}}. It follows from Sylvester algorithm, Example \ref{example: sylvester}.
		
		{\it Case {\rm (4b)}}. We may assume that $f = x_0^d +
                g(x_1,x_2)$, where $g$ is a binary form of rank
                $3$. Then, since $d\geq 3$, by \cite[Lemma
                1.12]{BBKT15}, we have that $f^\perp_2 = (y_1,y_2)_2
                \cap g^\perp_2$. If $d = 3$, the claim follows from
                \cite[Section 3]{CCO17}. If $d \ge 4$, since $g$ is a
                binary quartic of rank $3$, we have $g^\perp =
                (y_0,G_1,G_2)$, where the $G_i$'s are cubics, and,
                therefore, $f^\perp_2 = (y_0y_1,y_0y_2)$ and $h_{f}(2)=4$. Observe that, any set $\XX$ of $4$ points, non-collinear in $\PP^2$, has $\rho(\XX) = 2$. Hence, by Lemma \ref{lemma: regularity}, we have that $(I_\XX)_2 = f^\perp_2$ and the claim follows. In the case $d \geq 5$, the claim follows from Theorem \ref{thm: regularity}.
	
	{\it Case {\rm (4c)}}. If $d = 3$, it follows from \cite[Section 3]{CCO17}. If $d = 4$, by Lemma \ref{lemma: regularity}, for any minimal set of points $\XX$ apolar to $f$, we have that $f^\perp_2 = (I_\XX)_2$. Since by assumption there exists a minimal set of points $\XX$ which is a complete intersection of two conics, $(f^\perp)_2$ defines the {\it unique} minimal apolar set of points of $f$. If $d \geq 5$, then the claim follows from Theorem \ref{thm: regularity}.
	
	{\it Case {\rm (4d)}}. It follows from Lemma \ref{lemma: rank = essential}.
	\end{proof}	
	
	\begin{remark}
	From this result, we have that, given a polynomial of rank $4$, {\it all} minimal apolar set of points fall within the same configuration. In other words, we obtain a stratification of the space of rank $4$ polynomials accordingly to the configuration of minimal apolar sets of points. Moreover, we obtain that Question \ref{question:A} has a positive answer for polynomials of rank $4$. In particular, any minimal apolar set of points of a given rank $4$ polynomial have one of the following Hilbert functions:
	$$
		{\rm (4a)} : ~~ 1 ~~ 2 ~~ 3 ~~ 4 ~~ 4 ~~ \cdots; \quad {\rm (4b),~(4c)}: ~~ 1 ~~ 3 ~~ 4 ~~ 4 ~~ \cdots; \quad {\rm (4d)}: ~~ 1 ~~ 4 ~~ 4 ~~ \cdots.
	$$
	Note that, as explained in \cite[Section 3B.2]{eisenbud_geometry_2005}, the cases (4b) and (4c) can be distinguished by looking at finer numerical invariants related to their resolution as their {\it graded Betti numbers}. In particular, we have:
	$$
		{\rm (4b)} : 
		\begin{tabular}{ccc}
					1 & $\cdot$ & $\cdot$ \\
					$\cdot$ & $2$ & $1$ \\
					$\cdot$ & $1$ & $1$
		\end{tabular}
		~~~~~
		{\rm (4c)} : 
		\begin{tabular}{ccc}
					1 & $\cdot$ & $\cdot$ \\
					$\cdot$ & $2$ & $\cdot$ \\
					$\cdot$ & $\cdot$ & $1$
		\end{tabular}
	$$
	\end{remark}
	
	\begin{remark}
	As we said, we want to do that {\it step-by-step}, but it is not enough to pick a point in the Waring locus because, in order to construct the decomposition, we should also find the suitable coefficient to put in front of the power of the corresponding linear form.
		
		In the case (4b), for $d = 3,4$, we have seen that the scheme defined by the degree $2$ part of the apolar ideal of $f$ has a unique reduced point $P$ and then a non-reduced part (for $d = 3$) or a $1$-dimensional part (for $d = 4$). In both cases, we consider the linear form $\ell_P$ having the coordinates of $P$ as coefficients and we reduce the rank of $f$ by finding the suitable coefficient $c$ such that $f - c\ell_P^d$ has two essential variables, i.e., such that the $1$st catalecticant matrix has rank $2$. 
		
		In the case (4c), for $d = 3$, we can chose a random point $P$ in $\PP^2$. Then, we need to find a suitable coefficient such that $f - c\ell_P^3$ has rank $3$. In order to do so, we need to use equations for the space of (the closure of) rank $3$ polynomials. These equations are very difficult to find and, in general, are not always known. A list of known cases is nicely explained in \cite{LO13}. In the case of plane cubics of rank $3$, we know that this is a hypersurface given by the so-called {\it Aronhold invariant}. We explain this in Section \ref{sec: Aronhold}. 
	\end{remark}
	
\subsection{Polynomials of rank $5$.}
Possible configurations of $5$ points in projective space are in Figure \ref{fig: 5 points}.	\begin{figure}[h]
	\begin{center}
				\subfigure[(5a) Collinear]{
			\includegraphics[scale=0.3]{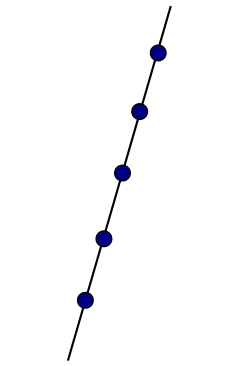}
		}
				\subfigure[(5b) Coplanar, with $4$ collinear.]{
		    \includegraphics[scale=0.28]{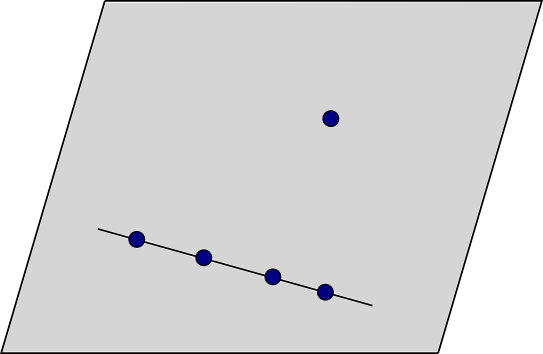}
		}
				\subfigure[(5c) Coplanar, on a unique conic]{
			\includegraphics[scale=0.28]{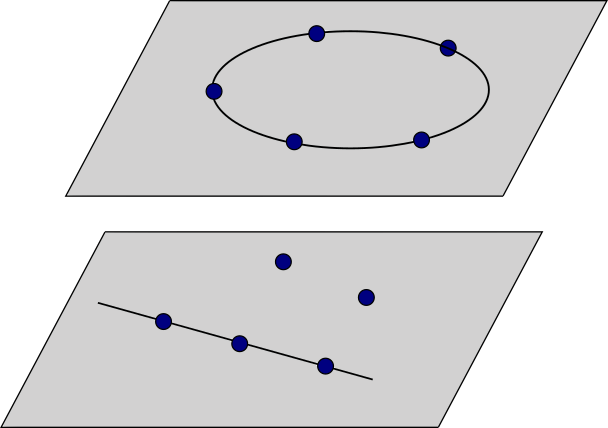}
		}	
				\subfigure[(5d) $4$ coplanar.]{
			\includegraphics[scale=0.28]{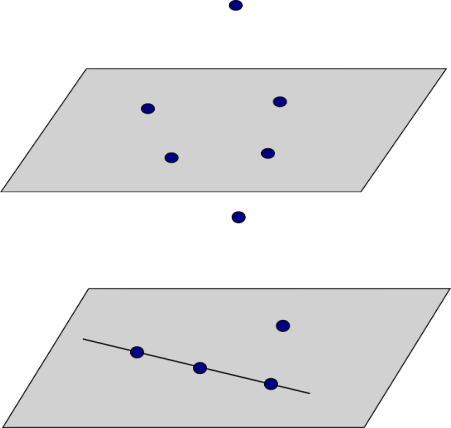}
		}	
			\subfigure[(5e) $5$ general in $\PP^3$.]{
			\includegraphics[scale=0.28]{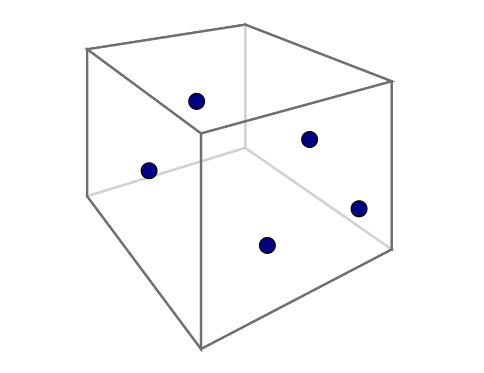}
		}	
			\subfigure[(5f) $5$ general in $\PP^4$.]{
			\includegraphics[scale=0.28]{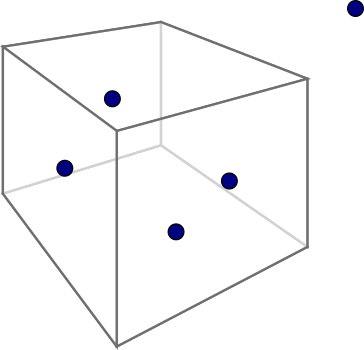}
		}	
		\end{center}
		\caption{Configurations of $5$ points in projective space.}\label{fig: 5 points}
	\end{figure}
	
	Since we have several cases to consider, we start with some preliminary lemma. In the first one, we study case (5b) in a more generality, by considering a set of $r$ points $\PP^2$ with $r-1$ collinear points.

	\begin{lemma}\label{lemma: general cusps rank5}
		Let $f = x_0^d + g(x_1,x_2) \in S_d$ of rank $\ge
                4$. Then, $\rk(f)=\rk(g)+1$ and we have that $\caW_f = (1:0:0) ~\cup~ \caW_g$. In other words, any minimal Waring decomposition of $f$ is given by the sum of $x_0^d$ and a minimal Waring decomposition of $g$, i.e., it is $x_0^d + \sum_{i=1}^r \ell_i^d(x_1,x_2)$, where $r = \rk(g)$. 
	\end{lemma}
	\begin{proof}
	Since binary forms of degree $d$ have rank at most $d$, we have that $r =\rk(g) \leq d$. 
	
	Now, by \cite[Lemma 1.12]{BBKT15}, we have $f^\perp_i = (y_1,y_2)_i \cap g^\perp_i$, for $i \leq d-1$. If $d = r$, we have that $g$ is a binary form of maximal rank $r$ which, up to a change of variables, can be written in the form $x_1x_2^{r-1}$. Hence, the claim follows from \cite[Theorem 5.1]{CCO17}. If $d \geq r+1$, then $g^\perp = (y_0,G_1,G_2)$ where $G_1$ has degree $r$ and $G_2$ has degree $d+2-r$. Since $r\geq 3$ and $d\geq r+1$, we have that the $G_i$'s have degree at least $3$. In particular, $f^\perp_2 = (y_0y_1,y_0y_2)$. Now, observe that any set of $r+1$ points  in $\PP^2$ with a subset of $r$ collinear points, have $\rho(\XX) \leq r-1$. Hence, by Lemma \ref{lemma: regularity}, for any minimal set of points $\XX$ apolar to $f$, we have $(I_{\XX})_i = f_i^\perp$, for $i \leq d-\rho(\XX)$. In particular, since $d-\rho(\XX) \geq d-r+1 \geq 2$, we have $(I_\XX)_2 = f^\perp_2 = (y_0y_1,y_0y_2)$. This concludes the proof.
	\end{proof}
	
	In the next lemma, we consider the case (5c) in the case of plane quartics.
	
	\begin{lemma}\label{lemma: plane quartics rank 5}
		Assume that $f$ is a plane quartic such that
                $\rk(f)=5$ and $f^\perp_2 = \langle C\rangle$. Then:
		\begin{enumerate}
			\item if $C$ is irreducible, then $\caW_f$ is dense in the conic $Z(C)$;
			\item if $C$ is reducible, say $C = L_1\cup L_2$, let $Q = L_1 \cap L_2  = Z(\ell_1)\cap Z(\ell_2)$; then:
			\begin{enumerate}
				\item if $Q$ {\it is not} a forbidden point for $f$, then $\caW_f$ is dense in $Z(C)$;
				\item if $Q$ {\it is} forbidden point for $f$, then, for either $i = 1$ or $i = 2$,  $\caW_f \cap Z(\ell_i)$ is dense in $Z(\ell_i)$.
			\end{enumerate}
		\end{enumerate}
	\end{lemma}
	\begin{proof}
		For any minimal set of points $\XX$ apolar to $f$, since by assumption we have that $\HF_{f}(2) = 5$, we have that $(I_{\XX})_2 = f^\perp_2 = \langle C \rangle$. Hence, we have that $\caW_f \subset Z(C)$.
		
		Let $\{U_1,\ldots,U_5\}$ be interpolation polynomials of $\XX$. Then, with respect to the basis $\langle U_1,\ldots,U_5,C \rangle$ of $T_2$, 
$$
	\cat_2(f) = \begin{bmatrix}
		I_5 & 0 \\
		* & 0 \\
	\end{bmatrix}.
$$
	For any point $P\in Z(C)$, the catalecticant matrix $\cat_2(\ell_P^d)$, with respect to the same basis as above, has the last column equal to $0$. We write
	$$
	\cat_2(\ell_P^4) = \begin{bmatrix}
		M_f & 0 \\
		* & 0 \\
	\end{bmatrix}.
$$
	Now, since $M_f$ is a rank $1$ symmetric matrix, there exists a (unique) non-zero eigenvalue $\lambda$ such that $\det({\rm Id} - \lambda M_f) = 0$. In particular, for such a $\lambda$, we have that $\cat_2(f+\lambda \ell_P^d)$ has rank $4$. Hence, we have a second conic $C'$ such that $(f+\lambda \ell_P^d)^\perp_2 = \langle C,C' \rangle$. Moreover, since the coefficients of $C'$ are rational polynomial functions in the coordinates of $P$, there exists a Zariski open subset $\calU$ in $Z(C)$, for which the conics $C$ and $C'$ meet transversally. We need to show when (and where) this open set is non-empty.
	
	\smallskip	(1) {\it If $C$ is irreducible, then $\calU$ is non-empty.} We know that there exists at least one minimal apolar set of points $\XX$ for $f$. This is a set of $5$ points lying on the irreducible conic $Z(C)$. In particular, if we assume $P$ to be one of these points, we have that $\XX \setminus \{P\}$ is a complete intersection of two conics. In particular, $P \in \calU$.
	
	\smallskip (2-a) {\it Let $C = L_1\cup L_2$, $Q = L_1 \cap L_2 = Z(\ell_1)
          \cap Z(\ell_2)$ and $Q$ is not a forbidden point for $f$.} By assumption, there exists a minimal apolar set of points $\XX$ for $f$ which includes the point $Q$.  By using this set of points, we can write $f = f_1 + f_2 + \ell_Q^4$, where $f'_1 = f_1+\ell_Q^4$ is a rank $3$ quartic in the two essential variables of the line $Z(\ell_1)$. By \cite[Theorem 3.5]{CCO17}, we know that $\caW_{f'_1}$ is dense in $Z(\ell_1)$. Since $\caW_{f'_1}\subset \caW_{f}$, we conclude that $\caW_f$ is dense in $Z(\ell_1)$. By proceeding in the same way with $f'_2 = f_2 + \ell_Q^4$, we conclude.
	
	\smallskip (2-b) {\it Let $C = \ell_1\ell_2$, $Q = Z(\ell_1) \cap Z(\ell_2)$ and $Q$ is a forbidden point for $f$.} By assumption, there exists a minimal apolar set of points $\XX$ which, since $C$ is reducible, splits as the union of three point on a line, say $Z(\ell_1)$, and two points on the other. Following a similar idea as above, we can write $f = f_1 + f_2$ where $f_1$ is a quartic in the two essential variables of the line $Z(\ell_1)$ of rank $3$. By \cite[Theorem 3.5]{CCO17}, we know that $\caW_{f_1}$ is dense in $Z(\ell_1)$ and this concludes the proof.
	\end{proof}

Now, we consider the case of cubics and quartics with four essential variables. 

\begin{lemma}\label{lemma: rank 5, deg 3, coplanar}
	Let $f$ be a \Alessandro{cubic of rank $5$ with four essential variables, i.e., $h_{f}(1)=4$,} and let $\XX$ be a minimal set of points apolar to $f$.
        Then:
	\begin{enumerate}
		\item if all but one point of $\XX$ are coplanar, no
                  three of them colinear, then there exists a unique
                  ternary cubic $f'$ and $c\neq 0$ such that $f = c\, x_0^3 +
                  f'(x_1,x_2,x_3)$, $\rk(f')=\rk(f)-1$, and any minimal Waring decomposition of $f$ is obtained from a minimal decomposition of $f'$, i.e., $\caW_f = (1:0:0:0) \cup \caW_{f'}$;
		\item if all but two points of $\XX$ are collinear
                  then $f = x_0^3 + x_1^3 + g(x_2,x_3)$, where $g$ is
                  a binary cubic of rank $3$, and any minimal Waring decomposition of $f$ is obtained from a minimal decomposition of $g$, i.e., $\caW_f = (1:0:0:0) \cup (0:1:0:0) \cup \caW_g$;
		\item otherwise, $f$ has a unique decomposition.
	\end{enumerate}
\end{lemma}
\begin{proof} Let $\XX$ be a minimal set of apolar points of $f$
  (with $|\XX|=5$).
  
If $\XX$ has not a subset of four coplanar
points, point (3) follows from the classical Sylvester's Pentahedral Theorem. We refer to \cite[Theorem 9.4.1]{Dol12} for a modern proof.
	
If all but one point of $\XX$ are coplanar points, then we
can write $f = x_0^3 + g(x_1,x_2,x_3)$ where $g$ is a ternary
cubic with  $4 \le \rk(g) \le 5$.

Then $f_{2}^{\perp}=(y_{0} y_{1}, y_{0}y_{2}, y_{0}y_{3})_{2}+
(g_{2}^{\perp}\cap \CC[y_{1}, y_{2}, y_{3}])_{2}$. As $\rk(g)\ge 4$, we 
have $Z(g_{2}^{\perp})=\emptyset$ and $Z(f_{2}^{\perp})= \{P\} \cup H$ where $P=(1:0:0:0)\in \XX$
is one of the apolar point and $H$ is the plane defined by the 
equation $x_{0}=0$ containing the other apolar points. 

By Sylvester's Pentahedral Theorem  (\cite[Theorem 9.4.1]{Dol12}) any minimal apolar set of points contains four coplanar points.
Then the non-coplanar point and the plane containing the 
other points are uniquely determined by $f_{2}^{\perp}$. 

Let $\ell$ be the linear form corresponding to the non-coplanar point
$P$ (or the isolated point of the zero locus of $f^\perp_2$).
In a suitable basis of $T_1$, we can write 
	$$
		\cat_1(f -c \ell^{3}) = \cat_1(f) - c\cat_1(\ell^3) = 
		\left[\begin{matrix}
			\overline{c}-c & 0 \\
			0 & \cat_1(f')
		\end{matrix}\right].
	$$
Therefore, there exists a unique value $c = \overline{c}$ for
which $\cat_1(f-c \ell^{3})$ has rank $3$ and a unique ternary cubic
$f'$ such that $f= c \, x_{0}^{3}+ f'(x_{1},x_{2},x_{3})$.
Consequently, $\caW_f = (1:0:0:0) \cup \caW_{f'}$. This proves (1).
        
If $\XX$ has three collinear points, then we can write $f = x_0^3 +
x_1^3 + g(x_2,x_3)$ where $g$ is a binary cubic of rank $3$. In this
case, we have that the zero locus of $f^\perp_2$ consists of two
points $(1:0:0:0)$ and $(0:1:0:0)$. Similarly as above, we conclude
that the points and the line where the three collinear points lie is
uniquely determined by $f_{2}^{\perp}$. This proves (2).
\end{proof}

\begin{lemma}\label{lemma: rank 5, deg 4, coplanar}
Let $f$ be a quaternary quartic of rank $5$ with $h_{f}(1)=4$ and let
$\XX$ be a minimal set of points apolar to $f$.
 \begin{enumerate}
 	\item if $\XX$ contains four coplanar points, then we may assume that $f = x_0^3 + g(x_1,x_2,x_3)$, where $g$ is a plane quartic of rank $4$ and we have that $\caW_f = (1:0:0:0) \cup \caW_g$;
 	\item otherwise, $f$ has a unique decomposition given
          by $I_{\XX}=(f_{2}^{\perp})$.
\end{enumerate}
\end{lemma}
\begin{proof}
If any $4$ of points $\XX$ are not coplanar, then $\XX$ has a regularity $\rho(\XX)=2$ and is defined by $5$ quadrics.
Then by Lemma \ref{lemma: regularity}, $(I_\XX)_2 = f^\perp_2$ and $f$
has a unique decomposition. This proves (2).

If the set $\XX$ contains four coplanar points, we may assume $f =
x_0^d + g(x_1,x_2,x_3)$, where $g$ is a ternary quartic of rank $4$.
Therefore, we know $f^\perp_2 = (y_1,y_2,y_3)_2 \cap g^\perp_2=
(y_0y_1,y_0y_2,y_0y_3)_2+ (g_{2}^{\perp}\cap \CC[y_{1},y_{2},y_{3}])$.
By Theorem \ref{prop: rank 4}(4b-ii \& 4c-ii), $h_{g}(2)=4$, 
$(g_{2}^{\perp}\cap \CC[y_{1},y_{2},y_{3}])$ is generated by two elements
$G_1,G_2$ defining either the $4$ points apolar to $g$ or a point $P$
apolar to $g$ and a line $L$ containing the $3$ other points apolar to
$g$.

This shows that $(1:0:0:0)$ is a point of any minimal set of points
apolar to $f$ and that  $\caW_f = (1:0:0:0) \cup \caW_g$, which proves (1).
\end{proof}

Now, we can give the complete description of rank $5$ polynomials.

\begin{theorem}\label{prop: rank 5}
	Let $f \in S_d$ be a form of rank $5$.
	\begin{enumerate}
	\item[{\rm (5a)}] If $f$ has two essential variables, then $f^\perp = (L_1,\ldots,L_{n-1},G_1,G_2)$, where $\deg(G_i) = d_i$ and $d_1 \leq d_2$. In particular, it has to be $d \geq 5$ and:
	\begin{enumerate}
		\item[(i)] if $d = 5,\ldots,8$, then $d_2 = 5$, and minimal apolar sets of points are defined by ideals $I_{\XX} = (L_1,\ldots,L_{n-1},HG_1+\alpha G_2)$, for a general choice of $H \in T_{8-d}$ and $\alpha \in \CC$;
		\item[(ii)] if $d \geq 9$, there is a unique minimal apolar set of points given by $I_{\XX} = (L_1,\ldots,L_{n-1},G_1)$.
	\end{enumerate}
	\item[{\rm (5b)}] If $f$ has three essential variables and a minimal apolar set $\XX$ of type {\rm (5b)}, then, $d \geq 4$ and:
			\begin{enumerate}
				\item[(i)] if $d = 4,5,6$, then any minimal apolar set is of the type $P \cup \XX'$, where $\XX'$ are collinear points;
				\item[(ii)] if $d \geq 7$, then $(f^\perp_{4})$ defines the unique minimal apolar set of points.
			\end{enumerate}
	\item[{\rm (5c)}] If $f$ has three essential variables and a minimal apolar set $\XX$ of type {\rm (5c)}, then:
	\begin{enumerate}
		\item[(i)] if $d = 3$, \Alessandro{then the Waring locus is dense in all $\PP^2$;}
		\item[(ii)] if $d = 4$, then, if $f^\perp_2 = \langle C \rangle$, we have:
		\begin{enumerate}
		\item[(a)] if $C$ is irreducible, then $\caW_f$ is dense in the conic $Z(C)$;
			\item[(b)] if $C$ is reducible, say $C = \ell_1\ell_2$, let $Q = Z(\ell_1)\cap Z(\ell_2)$; then:
			\begin{enumerate}
				\item[(b1)] if $Q$ {\it is not} a forbidden point for $f$, then $\caW_f$ is dense in $Z(C)$;
				\item[(b2)] otherwise, for either $i = 1$ or $i = 2$,  $\caW_f \cap Z(\ell_i)$ is dense in $Z(\ell_i)$.
			\end{enumerate}
		\end{enumerate} 
		\item[(iii)] if $d \geq 5$, then we have a unique minimal apolar set of points.
	\end{enumerate}
	\item[{\rm (5d)}] If $f$ has four essential variables and a minimal apolar set $\XX$ of type {\rm (5d)}, then it can be written as $f = x_0^d + g(x_1,x_2,x_3)$, where $g$ is a ternary form of rank four; then:
	\begin{enumerate}
		\item[(i)]  if $d = 3$, then $\caW_f = (1:0:0:0) \cup \caW_g$;
		\item[(ii)] if $d \geq 4$, then there is unique minimal apolar set given by $I_\XX = (f^\perp_2)$.
	\end{enumerate}
	\item[{\rm (5e)}] If $f$ has four essential variables and a minimal apolar set $\XX$ of type {\rm (5e)}, then:
	\begin{enumerate}
		\item[(i)] if $d = 3$, it is Sylvester Pentahedral Theorem and we have a unique decomposition in $f^\perp_2$;
		\item[(ii)] if $d \geq 4$, then there is a unique decomposition given by $I_\XX = (f^\perp_{2})$.
	\end{enumerate}	 
	\item[{\rm (5f)}] If $f$ has five essential variables, there is a unique minimal apolar set given by $I_{\XX} = (f^\perp_2)$.
	\end{enumerate}
\end{theorem}
\begin{proof}
	{\it Case {\rm (5a)}.} It follows from Sylvester algorithm, Example \ref{example: sylvester}. 
	
	{\it Case {\rm (5b)}.} Up to a change of coordinates, we may assume $f = x_0^d + g(x_1,x_2)$, where $g$ is a binary form of rank $4$ then, since binary forms have maximal rank equal to the degree, it has to be $d \geq 4$. The claim follows from Lemma \ref{lemma: general cusps rank5}, in the special case $r = 4$, and by Theorem \ref{prop: rank 4}(4a).
	
	{\it Case {\rm (5c)}.} If $d = 3$, it follows from \cite[Section 3.4]{CCO17}. If $d = 4$, it follows from Lemma \ref{lemma: plane quartics rank 5}. Since $5$ general points in $\PP^2$ have regularity $2$, it follows from Theorem \ref{thm: regularity} in the cases with $d \geq 5$.
	
	{\it Case {\rm (5d)}.} It follows from Lemma \ref{lemma: rank 5, deg 3, coplanar} and Lemma \ref{lemma: rank 5, deg 4, coplanar}.
	
	{\it Case {\rm (5e).}} If $d = 3$, this is the classical Sylvester Penthaedral Theorem \cite[Theorem 3.9]{OO13}. Since the regularity of $5$ points in $\PP^3$ is equal to $2$, by Lemma \ref{lemma: regularity}, if $d \geq 4$, we have that $(I_\XX)_2 = f^\perp_2$. Since five general points in $\PP^3$ are generated by quadrics, the claim follows. 

	{\it Case {\rm (5f)}.} It follows from Lemma \ref{lemma: rank = essential}.
\end{proof}

\begin{remark}
	Similarly as in the previous cases, this result gives us a stratification of polynomials of rank $5$. In particular, we obtain a positive answer to Question \ref{question:A}. Again, we want to underline that some of the cases cannot be distinguished just by looking at the Hilbert function of the set of points, but we should look at the entire resolution and, in particular, to the graded Betti numbers.
\end{remark}

%	\begin{remark}
%	Again, we get a stratification of the space of rank $4$ polynomials, in the sense that {\it all} minimal apolar sets of points are in the same configuration. In particular, any minimal apolar set of points of a rank $5$ polynomial share the same Hilbert function which is one of the following:
%	$$
%		{\rm (5a)} : ~~ 1 ~~ 2 ~~ 3 ~~ 4 ~~ 5 ~~ 5 ~~ \cdots; \quad {\rm (5b)}: ~~ 1 ~~ 3 ~~ 4 ~~ 5 ~~ 5 ~~ \cdots; \quad {\rm (5c)}: ~~ 1 ~~ 3 ~~ 5 ~~ 5 ~~ \cdots;
%	$$
%	$$
%		{\rm (5d),~(5e)} : ~~ 1 ~~ 4 ~~ 5 ~~ 5 ~~ \cdots; \quad {\rm (5f)}: ~~ 1 ~~ 5 ~~ 5 ~~ \cdots.
%	$$
%	\end{remark}

\section{Low rank symmetric tensor decomposition algorithm}\label{sec: main}

In this section, we summarize the low rank cases and give a
procedure, which determines the rank of the tensor when it is $\le 5$ and
computes its decomposition. The analysis depends on the {\it Hilbert sequence}
$h_{f}=[h_{f}(0), h_{f}(1), \ldots,h_{f}(d)]$ of $S/f^{\perp}$ and the locus of
$f^{\perp}_{i}= \ker \cat_{i}(f)$. We have $h_{f}(0)= h_{f}(d)=1$ and
the sequence $h_{f}$ is symmetric ($h_{f}(d-i)=h_{f}(i)$) of length
$d+1$ where $d=\deg(f)$.
The rank $r$ of $f$ is such that $r\ge l(f)=\max_{i}\{h_{f}(i)\}$.
Hereafter, we denote by  $*$ a finite sequence
of values of length at least $1$ and by $k^{*}$ a finite sequence of
constant terms $k$ of length at least $1$.

We consider here symmetric tensors of degree $d\ge 3$, since the
decomposition of quadrics can be done by rank decomposition of symmetric matrices. We implemente the procedure described in the following theorem in the algebra software {\it Macaulay2} \cite{M2}; see Section \ref{sec:M2}.
\begin{theorem}[\bf and low rank decomposition algorithm]\label{thm: main}
Let $f$ be a symmetric tensor of degree $d\ge 3$. Then, either one the
following points is satisfied or $\rk (f)> 5$:
{\small
\begin{center}
\begin{tabular}{l c l}
	\quad\quad {\begin{tabular}{c}{\sc Hilbert} \\ {\sc Sequence}\end{tabular}} & {\begin{tabular}{c}{\sc Extra} \\ {\sc Condition}\end{tabular}} &  {\sc Algorithm to find a minimal apolar set} \\
	\hline
	$(1)$ \quad $[1^*]$ & & $\rk(f)= 1$ and $(f_{1}^{\perp})$ defines the point apolar to $f$ \\
	\hdashline
	$(2)$ \quad $[1,2,*,2,1]$ & & $f$ has two essential variables and Sylvester algorithm is applied: \\
	& & {\rm (i)} if $f^\perp_{l(f)}$ defines a set of $l(f)$ reduced points, then $\rk(f) = l(f)$; \\
	& & {\rm (ii)} otherwise, $\rk(f) = d+2-l(f)$ and a minimal apolar set\\
	& & \quad  is given by the principal ideal generated by \\
	& & \quad  a generic form $g \in f^\perp_{d+2-l(f)}$\\
	\hdashline
	$(3)$ \quad $[1,3,3,1]$ & $Z(f^\perp_2) = \emptyset$ & a generic pair of conics $q_{1},q_{2}$ of $f_{2}^{\perp}$ defines $4$ points and $\rk(f)=4$\\
		 \hdashline
	$(4)$ \quad $[1,3,3,1]$ & \multirow{4}{*}{\begin{tabular}{c}$Z(f^{\perp}_{2}) = P \cup D$, \\ $P$ is simple point \\ $D$ connected, $0$-dim \\ $\deg(D) = 2$ \end{tabular}} & $\rk(f) = 4$ and $P$ is a point of any minimal apolar set; then, we find \\
	& & \quad \quad the scalar $c$ such that $f' = f-c\ell_P^3$ has two essential variables \\
	& & \quad \quad and we apply Sylvester algorithm to $f'$ as in {\rm (2)} \\
	& & \\
	\hdashline
	$(5)$ \quad $[1,3,3,1]$ & $Z(f^\perp_2) = D$ & $\rk(f) = 5$
        and, for a generic $P$ and a generic $c\neq 0$ such that \\
		 & $D$ connected, $0$-dim & \quad $f'=f+c\ell_P^{3}$ is a ternary cubic of rank $4$ and we apply {\rm (4)} to $f'$ \\
		 & $\deg(D) = 3$ & \\
	\hdashline
	$(6)$ \quad $[1,3,3^{*},3,1]$ & \multirow{2}{*}{\begin{tabular}{c}$Z(f^\perp_2) = \{P_1,P_2,P_3\}$ \\ $P_i$'s are simple points \end{tabular}} &
	$\rk(f) = 3$ and the unique minimal apolar set is $Z(f^\perp_2)$ \\
	& & \\
	\hdashline
	$(7)$ \quad $[1,3,*,3,1]$ & \multirow{3}{*}{\begin{tabular}{c}$Z(f^\perp_2) = P \cup L$ \\ $P$ is simple point \\ $L$ is line, $P \not\in L$ \end{tabular}} & $P$ is a point of any minimal apolar set; then, we find  \\
	& & \quad \quad the scalar $c$ such that $f' = f - c\ell_P^d$ has two essential variables \\
	& & \quad \quad  and we apply Sylvester algorithm to $f'$ as in {\rm (2)} \\
	\hdashline
	$(8)$ \quad $[1,3,4^*,3,1]$ & $Z(f^\perp_2) = \{P_1,\ldots,P_4\}$ & $\rk(f) = 4$ and the unique minimal apolar set is $Z(f^\perp_2)$ \\
	& $P_i$'s are simple points \\
	\hdashline
	$(9)$ \quad $[1,3,5,3,1]$ & $Z(f^\perp_2) = C$ & let $P$ be a generic point on $C$ and $c$ be a scalar such that \\ 
	& $C$ is irreducible quadric & \quad \quad $f' = f - c\ell_P^4$ has $h_{f'}(2) = 4$. \\
	& & {\rm (i)} if $Z((f')^\perp_2) = \{P_1,\ldots,P_4\}$ is a set of $4$ reduced points, then, \\
	& & \quad \quad $\rk(f) = 5$, and a minimal set apolar to $f$ is $\{P,P_1,\ldots,P_4\}$; \\
	& & {\rm (ii)} otherwise, $\rk(f) > 5$ \\
	\hdashline
	$(10)$ \quad $[1,3,5,3,1]$ & $Z(f^\perp_2) = L_1\cup L_2$ & let $P_i$ be a generic point on $L_i$, for $i = 1,2$, respectively, and \\
	& $L_i$'are distinct lines & \quad $c_i$ be a scalar such that $f_i = f - c_i \ell_{P_i}^4$ has $h_{f_i}(2) = 4$, for $i = 1,2$. \\
	& & {\rm (i)} if $Z((f_i^\perp)_2) = \{P_1,\ldots,P_4\}$, for either $i = 1$ or $i = 2$, then, \\
	& & \quad \quad $\rk(f) = 5$, and a minimal apolar set of $f$ is $\{P,P_1,\ldots,P_4\}$; \\
	& & {\rm (ii)} otherwise, $\rk(f) > 5$ \\
	\hdashline
	$(11)$ $[1,3,5,5^*,3,1]$ & $Z(f^\perp_3) = \{P_1,\ldots,P_5\}$ & $\rk(f) = 5$ and the unique minimal apolar set is $Z(f^\perp_3)$ \\
	& $P_i$'s are reduced points \\
	\hdashline
	$(12)$ \quad $[1,4,4,1]$ & $Z(f^\perp_2) = P \cup H$ & $P$ is a point of any minimal apolar set; then, we find \\
	& $P$ is a reduced point & \quad the scalar $c$ such that $f' = f - c\ell_P^3$ has three essential variables \\
	& $H$ is a plane, $P \not\in H$ & \quad and we apply $(3)$ or $(4)$ to $f'$ \\
	\hdashline
	$(13)$ \quad $[1,4,5^*,4,1]$ & $Z(f^\perp_2) = \{P_1,\ldots,P_5\}$ & $\rk(f) = 5$ and the unique minimal apolar set is $Z(f^\perp_2)$ \\
	\hdashline
	$(14)$ \quad $[1,5,5^*,5,1]$ & $Z(f^\perp_2) = \{P_1,\ldots,P_5\}$ & $\rk(f) = 5$ and the unique minimal apolar set is $Z(f^\perp_2)$ \\
\end{tabular}
\end{center}
}
\end{theorem}

\begin{proof}
By the analysis of the previous sections, a symmetric tensor of rank
$\le 5$ satisfies one of these cases.

Let us prove conversely that if one of these cases is satisfied then the rank is determined.

  {\em  Case $(1)$.} $f$ has one essential variable and
  thus $\rk(f)=1$.

  {\em Case $(2)$.} $f$ has two essential variables and can be decomposed by Sylvester algorithm; see Example \ref{example: sylvester}.

 {\em  Case $(3), (4)$ and $(5)$.} We have $\rk(f)\ge h_{f}(1)=3$.
  If $\rk(f)=3$, then by \Alessandro{Proposition} \ref{prop: rank 3}, $f_{2}^{\perp}$ should define $3$ reduced points, which is not the case.
  Hence, \Alessandro{since the maximal rank of plane cubics is $5$, we have} $4\le \rk(f)\le 5$.
  \Alessandro{Considering the classification of plane cubics (see \cite{LT10}), it is possible to check that we have three possibilities for $Z(f^\perp_2)$.

  If $Z(f^\perp_2) = \emptyset$, we know that the Waring locus is dense in the projective plane. If $\rk(f)=4$, for a generic point $P$, there exists $P_{1},P_{2},P_{3}$ such that $\XX=\{P, P_{1}, P_{2}, P_{3}\}$ is a minimal set of points apolar to $f$. Moreover, they are of type (4c). Then,
  $(I_{\XX})_{2}\subset f_{2}^{\perp}$ is spanned by two quadrics
  $q_{1},q_{2}\in f_{2}^{\perp}$. Thus $(I_{\XX})_{2}$ is the linear
  space of quadrics in $f_{2}^{\perp}$ containing $P$.  Conversely, a generic subset of $f_{2}^{\perp}$ of dimension $2$ is the space of quadrics in $f_{2}^{\perp}$ containing a generic point
  $P$. It coincides with $(I_{\XX})_{2}$ for a minimal set of
  points $\XX$ apolar to $f$. This proves the point (3).

 If $Z(f^\perp_2) = P \cup D$, where $P$ is a simple point and $D$ is a degree $2$ connected $0$-dimensional scheme. As $f$ has three essential variables, we can assume that $f$ is a ternary form in the variables $x_0,x_1,x_2$. By a change of coordinates, $P = (1:0:0)$ and $D$ lies on the line $y_0 = 0$, e.g., $D$ is defined by the ideal $(y_0,y_1^2)$. Then, $f = x_0^3 + f'(x_1,x_2)$. Since binary cubics have rank at most $3$, we have $\rk(f) = 4$. By Theorem \ref{prop: rank 4}(4b-i), $P$ is in any minimal set apolar to $f$. The other three (collinear) points to get a minimal set of points apolar to $f$ are found by applying (2) to the form $f' = f - c\ell_P^3$, where $c$ is a suitable scalar such that $f'$ has two essential variables, i.e., such that $\rk(\cat_1(f) - c\cat_1(\ell_P^3)) = 2$. This proves the point (4).
 
 If $Z(f^\perp_2) = D$, where $D$ is a degree $3$ connected $0$-dimensional scheme lying on a plane conic. Since the rank $4$ cases have $Z(f^\perp_2)$ which is either empty or the union of a simple point and a degree $2$ scheme, we have that $\rk(f) = 5$. Then, by Theorem \ref{thm: Waring loci supgeneric}, for any generic point $P$ and any non-zero $c\in
  \CC$, $f'= f+ c\ell_P^{3}$ is of rank $4$ and the previous
  decomposition applies.  This proves the point (5).}

{\em  Cases $(6), (8), (11), (13)$ and $(14)$.} They are consequences of Lemma \ref{lemma:fk}. 

 {\em  Case $(7)$.} As $f$ has $3$ essential variables, we can assume
  that it is a ternary form in the variables $x_{0},x_{1},x_{2}$.
  By a change of coordinates, we can also assume that
  $f_{2}^{\perp}$ defines $P=(1:0:0)$ and the line $L$
  of equation $y_{0}=0$.
  Then, \Alessandro{$f$ can be written (up to scalar) as $f=x_{0}^{d}+f'(x_{1},x_{2})$} and $f_{2}^{\perp}\supset \langle y_{0} y_{1}, y_{0} y_{2}\rangle$.

  If $\rk(f)=3$, then by $f_{2}^{\perp}$ should define the apolar
  points, which is not the case. 
  Hence, $\rk(f)\ge 4$ \Alessandro{and we} deduce the result from Lemma \ref{lemma:
  general cusps rank5}. \Alessandro{In particular, $P$ is in any minimal set of points apolar to $f$. The other (collinear) points to get a minimal set of points apolar to $f$ are found by applying (2) to the form $f' = f - c\ell_P^3$, where $c$ is a suitable scalar such that $f'$ has two essential variables, i.e., such that $\rk(\cat_1(f) - c\cat_1(\ell_P^3)) = 2$.}

  {\em  Cases $(9)$ and $(10)$.}
We have $\rk(f)\ge h_{f}(2)=5$.
If $\rk(f)=5$, we deduce the decomposition of $f$ by applying Lemma
\ref{lemma: plane quartics rank 5}. \Alessandro{In particular, choosing one of the intersections between a generic line and the irreducible conic $Q$, in the case (8), or the reducible conic $L_1L_2$, in the case (9), we can find a scalar $c$ such that $f' = f - c\ell_P^4$ has rank $4$, i.e., such that $\rk(\cat_2(f) - c\cat_2(\ell_P^4)) = 4$. Then, we apply (7) to $f'$.}

 {\em  Case $(12)$.} As $f$ has $4$ essential
  variables, \Alessandro{we can assume it} is a quaternary cubic in the variables
  $x_{0},x_{1},x_{2}, x_{3}$. By a change of coordinates, we can
  assume that the zero locus of $f_{2}^{\perp}$ is $P=(1:0:0:0)$ and
  the plane $H$ defined by $y_{0}=0$.
  Then, \Alessandro{$f$ can be written (up to a scalar) as
  $f=x_{0}^{d}+f'(x_{1},x_{2}, x_{3})$} and $f_{2}^{\perp}\supset\langle y_{0} y_{1}, y_{0} y_{2}, y_{0} y_{3}\rangle$. \Alessandro{Since ternary cubics have rank at most $5$,} we have $4=h_{f}(1)\le \rk(f)\le 6$. If $\rk(f)=4$, by Theorem \ref{prop: rank 4}(5d),
  $f_{2}^{\perp}$ should define $4$ reduced points, which is not the
  case. Thus $\rk(f)\ge 5$.

  If $\rk(f)=5$, by Lemma \ref{lemma: rank 5, deg 3, coplanar}, $P$ is
  a point of any  minimal apolar set of points of $X$, $\rk(f)=\rk(f')+1$ and the other
  points form a minimal set of points apolar to $f'$.

  If $\rk(f)=6$, then $\rk(f')=5$, $P$ is one of the apolar points to
  $f$ and the other are the apolar points to $f'$.

  These points can be computed by finding the scalar $c$ such that $f'=f-c\ell_
  P^{d}$ has 3 essential variables, \Alessandro{i.e., by imposing $\rank (\cat_{1}(f)-c\,
  \cat_{1}(\ell_P^{d}))=3$,} and by applying (3) and (4) to the cubic $f'$. 
\end{proof}
 
\section{A {\it Macaulay2} package}\label{sec:M2}

\Alessandro{The procedure explained in the previous section} can be implemented by using computational algebra or computer algebra softwares. We chose to use the algebra software {\it Macaulay2} \cite{M2}. The package {\tt ApolarLowRank.m2} here described can be found on the personal webpage of the second author \Alessandro{or in the ancillary files of the arXiv and HAL versions of the article and can be loaded as}
{\small
\begin{verbatim}
   i1 : loadPackage "ApolarLowRank"
   o1 = ApolarLowRank
   o1 : Package
\end{verbatim}
}
\noindent \Alessandro{For more details, we refer to the documentation }
 {\small
\begin{verbatim}
   i2 : viewHelp ApolarLowRank
\end{verbatim}
}
\noindent \Alessandro{In the following, we explain some of the main functions and we show how it works in a few examples.

\subsection{Essential variables} 
As we have explained in Section \ref{ssec: essential}, given a homogeneous polynomial $f \in S$, the essential number of variables of $f$ is the smallest number $N$ such that there exists linear forms $\ell_1,\ldots,\ell_N\in S$ such that $f \in \CC[\ell_1,\ldots,\ell_N]$. In our packege, we have implemented the functions:}
\begin{itemize}
\item {\tt essVar}, which returns the number of variables of $f$ and a list of linear forms generating $f^\perp_1$;
{\small
\begin{verbatim}
   i3 : S = QQ[x,y,z,t];
   i4 : F = (x+y)^5 + (z-t)^5;
   i5 : essVar(F)
   o5 = (2, {- x + y, z + t})
   o5 : Sequence
\end{verbatim}
}
\item {\tt simplifyPoly}, which returns a simplified version of the polynomial in a set of essential variables and a ring map describing the linear change of coordinates needed.
{\small
\begin{verbatim}
   i6 : simplifyPoly(symbol Y, F)
          5    5
   o6 = (Y  - Y , map(S,QQ[Y , Y ],{x + y, - z + t}))
          0    1            0   1
   o6 : Sequence
\end{verbatim}
}
\end{itemize}
\Alessandro{Note that as input in the function {\tt simplifyPoly} it is required also a {\tt Symbol} so that the user can chose a name for the indexed variables for the output. 

\subsection{Two essential variables: Sylvester's algorithm}
In the case of two essential variables, Sylvester's algorithm tells us how to find a minimal set of points apolar to a given form; see Example \ref{example: sylvester}.

In our package, we implemented the function {\tt sylvesterApolar} that returns a minimal set of points apolar to a given form with two essential variables.

Note that, as input, it is required also a {\tt Symbol} so that the user can chose a name for the indexed variables for the output, which is expressed in a set of essential variables of the polynomial. The output is a {\tt ApolarScheme} which is new type of {\tt HashTable} that we have introduced within the package. In particular, an {\tt ApolarScheme} has four attributes: 
\begin{enumerate}
\item {\tt hPoly}, which is a homogeneous polynomial; 
\item {\tt idX}, which is the ideal defining a $0$-dimensional scheme apolar to the polynomial given by {\tt hPoly}; 
\item {\tt Xdeg}, which is an integer giving the degree of the $0$-dimensional scheme; 
\item {\tt Xred}, which is a boolean saying if the $0$-dimensional scheme is whether reduced or not.
\end{enumerate}
Hence, the function {\tt sylvesterApolar} works as follows.}
{\small
\begin{verbatim}
   i7 : sylvesterApolar(symbol Y, F)
                                5    5
   o7 = (ApolarScheme{hPoly => Y  - Y   }, map(S,QQ[Y , Y ],{x + y, - z + t}))
                                0    1               0   1
                      idX => ideal(Y Y )
                                    0 1
                      Xdeg => 2
                      Xred => true
   o7 : Sequence
\end{verbatim}
}

\subsection{Ternary cubics}\label{sec: Aronhold}
\Alessandro{The cases of homogeneous polynomials with three essential variables are dealt with the function {\tt planar5Apolar}. Here, we want to explain how we implemented the cases of ternary cubics, i.e., the cases (3), (4) and (5). As we have seen, the distinction between these cases is given by the vanishing locus of $f^\perp_2$. }
{\small
\begin{verbatim}
   i8 : S = QQ[x,y,z];
   i9 : F = random(3,S);              -- case (3)
   i10 : G = random(QQ)*x^3 + y*z^2;   -- case (4)
   i11 : H = x*y^2 + y*z^2;            -- case (5)  
   -- Consider the degree 2 part of the apolar ideal
   i12 : Fperp2 = ideal(select(first entries gens perpId(F), i->degree(i)=={2}));
   o12 : Ideal of S
   i13 : Gperp2 = ideal(select(first entries gens perpId(G), i->degree(i)=={2}));
   o13 : Ideal of S
   i14 : Hperp2 = ideal(select(first entries gens perpId(H), i->degree(i)=={2}));
   o14 : Ideal of S   
   -- Check the properties of the corresponding vanishing locus
   i15 : dim Fperp2
   o15 = 0  
   i16 : primaryDecomposition Gperp2
                     2
   o16 = {ideal (x, y ), ideal (y, z)}
   o16 : List
   i17 : primaryDecomposition Hperp2, radical Hperp2
                              2   2
   o17 = ({ideal (x*z, x*y - z , x )}, ideal (z, x))
   o17 : Sequence
\end{verbatim}
}
\Alessandro{First, we consider the case (3). The general plane cubic has rank $4$ and, as explained in \cite[Section 3.4]{CCO17}, we know that the Waring locus is dense in the whole plane of ternary linear forms. In other words, given a random ternary cubic $\ttF$ and a random ternary linear form ${\tt L}$, there exists a coefficient ${\tt c}$ such that the cubic ${\tt F'}$ defined as ${\tt F - c*L^3}$ has rank $3$. Then, ${\tt F'}$ has a unique decomposition which is easy to compute. In order to compute the suitable value of ${\tt c}$, we need to intersect the line spanned by $\ttF$ and the third power of ${\tt L}$ with the (Zariski closure) of the space of plane cubics of rank $3$ (see Remark \ref{remark: Aronhold invariant}).

This is a function to compute the Aronhold invariant of a given cubic with three essential variables.
}
{\small
\begin{verbatim}
   aronhold = method();
   aronhold (RingElement) := F -> (
       R := ring F;    V := (entries vars R)_0;
       K := matrix{{0,-V_2,V_1},{V_2,0,-V_0},{-V_1,V_0,0}};
       C := diff(basis(1,R), transpose diff(basis(1,R),F));
       KF := diff(K,C);
       Pf := pfaffians(8,KF); 
       if Pf != sub(ideal (),R) then return Pf_0 else return 0_R
    )
\end{verbatim}
}
\Alessandro{Now, we can find the suitable coefficient to reduce the rank of the general cubic $\ttF$ in $S = {\tt QQ[x,y,z]}$. }
{\small
\begin{verbatim}
   i19 : L = random(1,S)
         2     7    7
   o19 = -x + --y + -z
         7    10    9
   o19 : S
   i20 : R = QQ[c][x,y,z];
   i21 : F' = sub(F,R) - c*(sub(L,R))^3;
   i22 : Ic = ideal aronhold F'
                 1217672402543    305183
   o22 = ideal(- -------------c + ------)
                   2083725000       125
   o22 : Ideal of R
   i23 : F' = sub(sub(F',R/Ic),S);
\end{verbatim}
}
\Alessandro{Hence, ${\tt F'}$ has rank $3$ and a unique decomposition, which is given by $Z((f')^\perp_2)$. By adding the point corresponding to the form ${\tt L}$, we conclude and find the ideal ${\tt IX}$ of a minimal set of points apolar to $\ttF$.}
{\small
\begin{verbatim}
   i24 : IP = ideal(basis(1,S) * gens kernel transpose (coefficients(L))_1)
   o24 = ideal (- 49x + 20y, - 49x + 18z)
   o24 : Ideal of S
   i25 : IX' = ideal(select(first entries gens perpId(F'), i->degree(i)=={2}));
   o25 : Ideal of S
   i26 : IX = intersect(IX',IP);
   o26 : Ideal of S
   -- Check if IX given a minimal set of points apolar to F
   i27 : dim IX, degree IX, IX == radical IX, isSubset(IX,perpId(F))
   o27 = (1, 4, true, true)
   o27 : Sequence
\end{verbatim}
}
Now, we consider the case (4). 
{\small
\begin{verbatim}
   i28 : G = random(QQ)*x^3+y*z^2
        5 3      2
   o28 = -x  + y*z
        8
   o28 : S
   i29 : R = QQ[c][x,y,z];
   i30 : G' = sub(G,R) - c*x^3;
   i31 : Ic = trim minors(3,cat(1,G'))
   o31 = ideal(8c - 5)
   o31 : Ideal of R
   i32 : G' = sub(sub(G',R/Ic),S);
   i33 : perpId G'
                   2   3
   o33 = ideal (x, y , z )
   o33 : Ideal of S
   -- Use Sylvester's Algorithm to find a minimal set apolar to G'
   i34 : IX' = ideal(x,z^3 - (random(QQ)*y+random(QQ)*z)*y^2);
   o34 : Ideal of S
   -- Adding the point (1:0:0) corresponding to the linear form `x', we conclude
   i35 : IX = intersect(IX',ideal(y,z));
   o35 : Ideal of S
   -- Check if IX given a minimal set of points apolar to G
   i36 : dim IX, degree IX, IX == radical IX, isSubset(IX,perpId(G))
   o36 = (1, 4, true, true)
   o36 : Sequence
\end{verbatim}
}
\Alessandro{As regards the case (5), if we consider a ternary cubic {\tt H} of maximal rank $5$, as explained in \cite[Section 3.4]{CCO17}, we can use a random linear form {\tt L} to reduce the rank. Then, we apply the previous cases. We can check that the rank of {\tt H-L\^{}3} drops by looking at the degree $2$ part of the apolar ideal, as explained.}
{\small
\begin{verbatim}
   i37 : H = x*y^2 - y*z^2
            2      2
   o37 = x*y  - y*z
   i38 : H' = H - L^3;
   i39 : Hperp2' = ideal(
                      select(first entries gens perpId(H'), i->degree(i)=={2})
                   )
                   2                    2               2   2
   o39 = ideal (24y  + 5x*z - 12y*z - 6z , x*y - x*z + z , x  - 2x*z)
   o39 : Ideal of S
   i40 : dim Hperp2'
   o40 = 0
\end{verbatim}
}
\subsection{Rank $5$ plane quartics} 
\Alessandro{Here, we want to comment the cases (9) and (10) of plane quartics $f$ having $h_f(2) = 5$. If the unique apolar conic $C$ is irreducible (case (9)), then we can reduce the rank of $f$ by taking a generic point on the conic; see Lemma \ref{lemma: plane quartics rank 5}(a). In our implementation, this is done by considering the intersections of a generic line and the conic $C$. Since this involves solving a quadratic equation which might not have solution of {\tt QQ}, in this case, the output of our main function {\tt minimalApolar5} depends on a parameter satisfying that quadratic equation and, for this reason, the ideal {\tt idX} has degree $10$ instead of $5$.}
{\small
\begin{verbatim}
  i41 : S = QQ[x,y,z];
  i42 : F = sum for i to 4 list (random(1,S))^4;
  i43 : first minimalApolar5(symbol Y, F)
                                         4              3                2 2   
  o43 = ApolarScheme{hPoly =>14000508577Y +118257495840Y Y +394649884800Y Y ...
                                         0              0 1              0 1 
                     idX => ideal (42277476088772685406212514432670091701648...
                     Xdeg => 10
                     Xred => true
   o43 : ApolarScheme              
   i44 : ring X#idX
              QQ[a][Y , Y , Y ]
                     0   1   2
   o44 = ----------------------------
         2
        a  - 31293683294493204311665
  o44 : QuotientRing                                                                                                       
\end{verbatim}
}
\Alessandro{When the unique apolar conic is reducible $C = L_1L_2$ (case (10)), we proceed in a similar way as in the previous case. However, it is not enough to consider a generic point on $C$ because, as we said in Lemma \ref{lemma: plane quartics rank 5}(2), if the intersection point $Q = L_1 \cap L_2$ is forbidden for the form $f$, then the Waring locus is dense only in one of the two lines. We see that in an example which also shows how we implemented the procedures explained in Theorem \ref{thm: main}(9-10) to find a minimal set of points apolar to $f$.

We consider an example where $f^\perp_2 = (xz)$.}
{\small
\begin{verbatim}
   i56 : G = x^2*y^2;
   i57 : L1 = random(QQ)*y + random(QQ)*z;
   i58 : L2 = random(QQ)*y + random(QQ)*z;
   i59 : F = L1^4 + L2^4 + G;
   i60 : perpId F
                         2           2       3        3            2         
   o60 = ideal (x*z, 648y z - 1467y*z  + 560z , 46656y  - 198801y*z  + ...
   o60 : Ideal of S
\end{verbatim}
}
Now, we consider a random point on the line $\{x = 0\}$.
{\small
\begin{verbatim}
   i61 : L = random(QQ)*y + random(QQ)*z
         3
   o61 = -y + 10z
         8
   o61 : S
   i62 : R = QQ[c][x,y,z];
   i63 : F1 = sub(F,R) - c*sub(L^4,R);
   i64 : Ic = radical minors(5,cat(2,F1))
   o64 = ideal(- 10368345145825c + 717382656)
   o64 : Ideal of R
   i65 : F1 = sub(sub(F1,R/Ic),S);
   i66 : F1perp2 = ideal select(first entries gens perpId F1,i->degree(i)=={2})
                                 2                                2
   o66 = ideal (x*z, 84009951144x  - 26206863345y*z + 26903620240z )
   o66 : Ideal of S
   i67 : dim F1perp2, degree F1perp2
   o67 = (1, 4)
   o67 : Sequence
   i68 : netList primaryDecomposition F1perp2
         +---------------------------------------------+
   o68 = |ideal (5241372669y - 5380724048z, x)         |
         +---------------------------------------------+
         |        2                  2                 |
         |ideal (z , x*z, 9334439016x  - 2911873705y*z)|
         +---------------------------------------------+
\end{verbatim}
}
Since {\tt F1} has the vanishing locus of the homogeneous part of degree $2$ which is not among the cases listed in Theorem \ref{thm: main}, we conclude that it has rank $6$. Hence, the random point on the line $\{x = 0\}$ is not in the Waring locus of {\tt F}. Now, we proceed by considering a random point on the line $\{z = 0\}$.
{\small
\begin{verbatim}
   i69 : L = random(QQ)*x + random(QQ)*y
         9    1
   o69 = -x + -y
         2    3
   o69 : S
   i70 : R = QQ[c][x,y,z];
   i71 : F2 = sub(F,R) - c*sub(L^4,R);
   i72 : Ic = radical minors(5,cat(2,F2))
   o72 = ideal(- 81c + 2)
   o72 : Ideal of R
   i73 : F2 = sub(sub(F2,R/Ic),S);
   i74 : F2perp2 = ideal select(first entries gens perpId F2,i->degree(i)=={2})
                        2                 2                   2
   o74 = ideal (x*z, 32x  + 432x*y + 5832y  - 13203y*z + 5040z )
   o74 : Ideal of S
   i74 : dim F2perp2, degree F2perp2, F2perp2 == radical F2perp2
   o74 = (1, 4, true)
\end{verbatim}
}
Hence, {\tt F2} is a quartic of rank $4$ whose apolar ideal in degree $2$ defines a minimal apolar set of points. Hence, we we can find a minimal set of $5$ points apolar to {\tt F}.
{\small
\begin{verbatim}
   i75 : IP = ideal(basis(1,S) * gens kernel (diff(vars S,L)))
   o75 = ideal (- 2x + 27y, z)
   o75 : Ideal of S
   i76 : IX = intersect(IP,F2perp2)
                         2           2       3      3           3             
   o76 = ideal (x*z, 648y z - 1467y*z  + 560z , 512x  - 1259712y  + ...
   o76 : Ideal of S
   i77 : dim IX, degree IX, IX == radical IX, isSubset(IX,perpId(F))
   o77 = (1, 5, true, true)
   o77 : Sequence
\end{verbatim}}
\subsection{Main function}
\Alessandro{All procedures listed in Theorem \ref{thm: main} have been collected in the function {\tt minimalApolar5} that produces a minimal set of points apolar to a given polynomial of rank at most $5$ in any number of variables and any degree by using the suitable algorithm, as explained in Theorem \ref{thm: main}. Here, we want to present some tests we made by using a personal computer with processor Intel Core i7 with 2,2 GHz. 

We first tested the efficiency of the main function in relation to the number of essential variables. In particular, we considered five different cases of minimal apolar set:
\begin{enumerate}
\item $5$ generic points in a $\PP^4$;
\item $5$ generic points in a $\PP^3$;
\item $4$ generic coplanar points plus a generic point;
\item $3$ collinear points plus two generic points;
\item $5$ generic points in a $\PP^2$.
\end{enumerate}
Here is the code used:}
\begin{verbatim}
   -- fix:   d = degree;   n = number of variables
   S = QQ[x_0..x_n];
   -- fix the "essential variables"
   L = for i to 4 list random(1,S);
   -- A) five generic points in P^4 
   F = sum for i to 4 list L_i^d;
   -- B) five generic points in P^3
   G = sum for i to 4 list (L_0 + random(QQ)*L_1 + random(QQ)*L_2 + random(QQ)*L_3)^d;
   -- C) four generic points in P^2 + 1 generic point
   H = L_0^d + sum for i to 3 list (L_1 + random(QQ)*L_2 + random(QQ)*L_3)^d;
   -- D) three collinear points in P^1 + 2 generic points
   K = L_0^d + L_1^d + sum for i to 2 list (random(QQ)*L_2 + random(QQ)*L_3)^d;
   -------- if d = 3: K = L_0^3 + L_1^3 + L_2*L_3^2;
   -- E) five generic points in P^2
   M = sum for i to 4 list (random(QQ)*L_0 + random(QQ)*L_1 + random(QQ)*L_2)^d;
   -------- if d = 3: M = L_0*L_1^2 + L_1*L_2^2;
   time minimalApolar5(symbol Y, F)
   time minimalApolar5(symbol Y, G)
   time minimalApolar5(symbol Y, H)
   time minimalApolar5(symbol Y, K) 
   time minimalApolar5(symbol Y, M)
\end{verbatim}

After fixing the degree $d = 3,4,5$, we let the number of essential variables grow. The tables in Figure \ref{fig: deg3}, \ref{fig: deg4} and \ref{fig: deg5} describe the time needed for our computations.
{\begin{figure}[H]
{\includegraphics[scale=0.47]{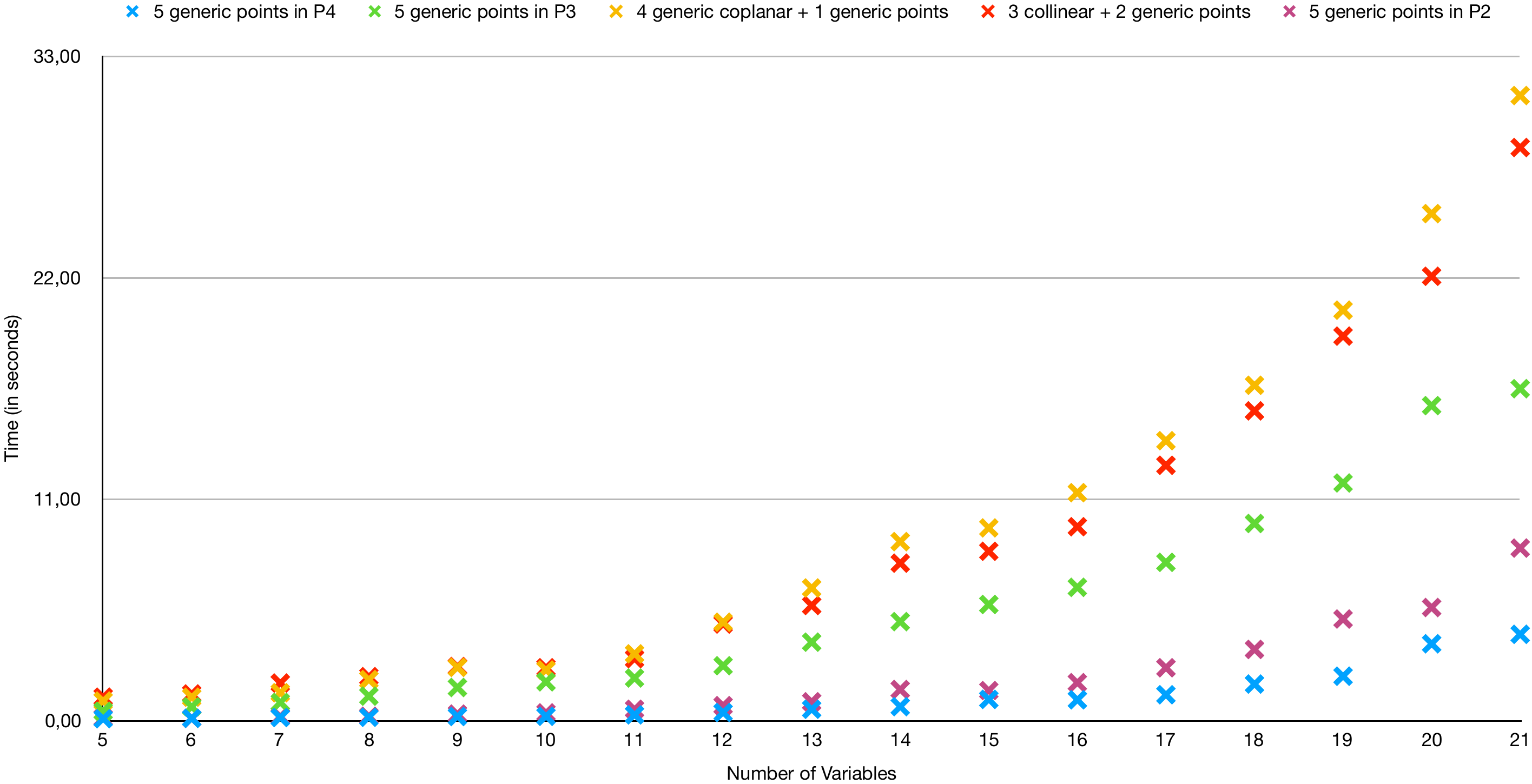}}
\caption{Tests with fixed degree equal to $3$.}
\label{fig: deg3}
\end{figure}
\begin{figure}[H]
{\includegraphics[scale=0.47]{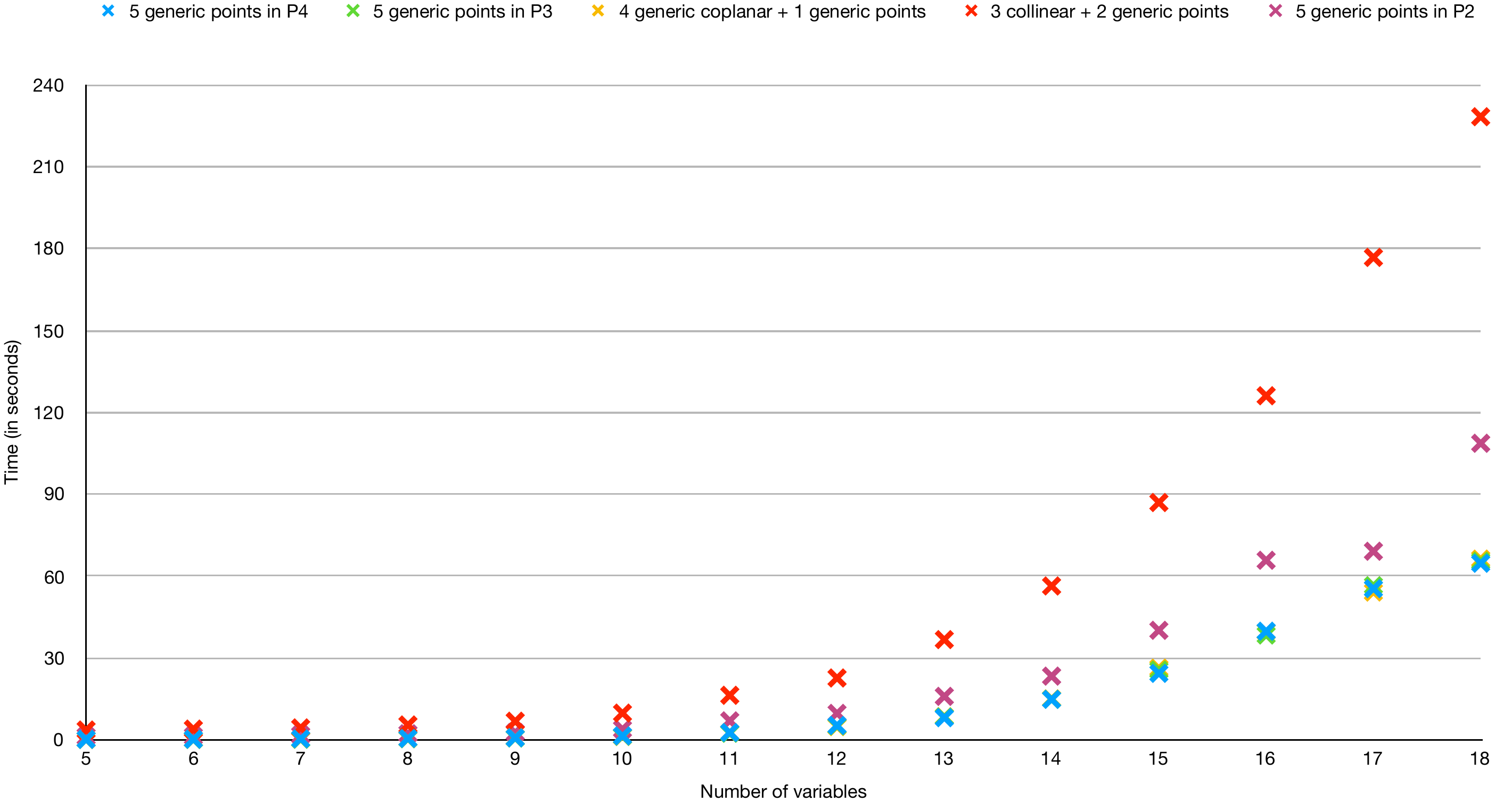}}
\caption{Tests with fixed degree equal to $4$.}
\label{fig: deg4}
\end{figure}\begin{figure}[H]
{\includegraphics[scale=0.47]{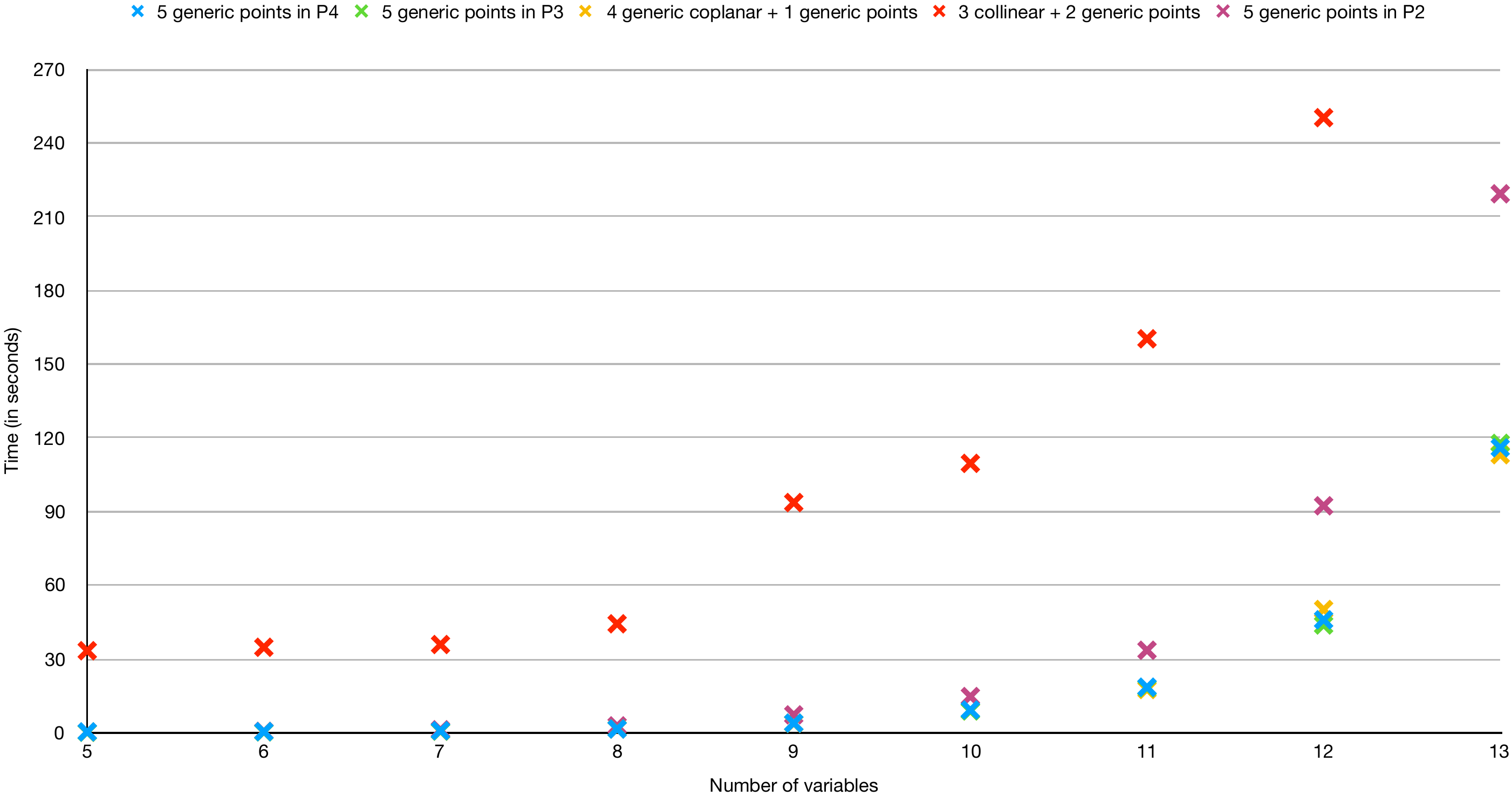}}
\caption{Tests with fixed degree equal to $5$.}
\label{fig: deg5}
\end{figure}}

We want to underline that the first step of the function {\tt minimalApolar5} is to reduce the polynomial in a minimal set of variables. This is the reason why the the function works quite efficiently also in a large set of variables. It seems that the complexity of our function {\tt minimalApolar5} depends more on the degree of the polynomial: we tested the same cases as before by fixing the number of variables ($n = 5$) and by letting the degree grow. Here is the table describing the time needed for our computations; see Figure \ref{fig: vars6}.

\begin{figure}[H]
{\includegraphics[scale=0.47]{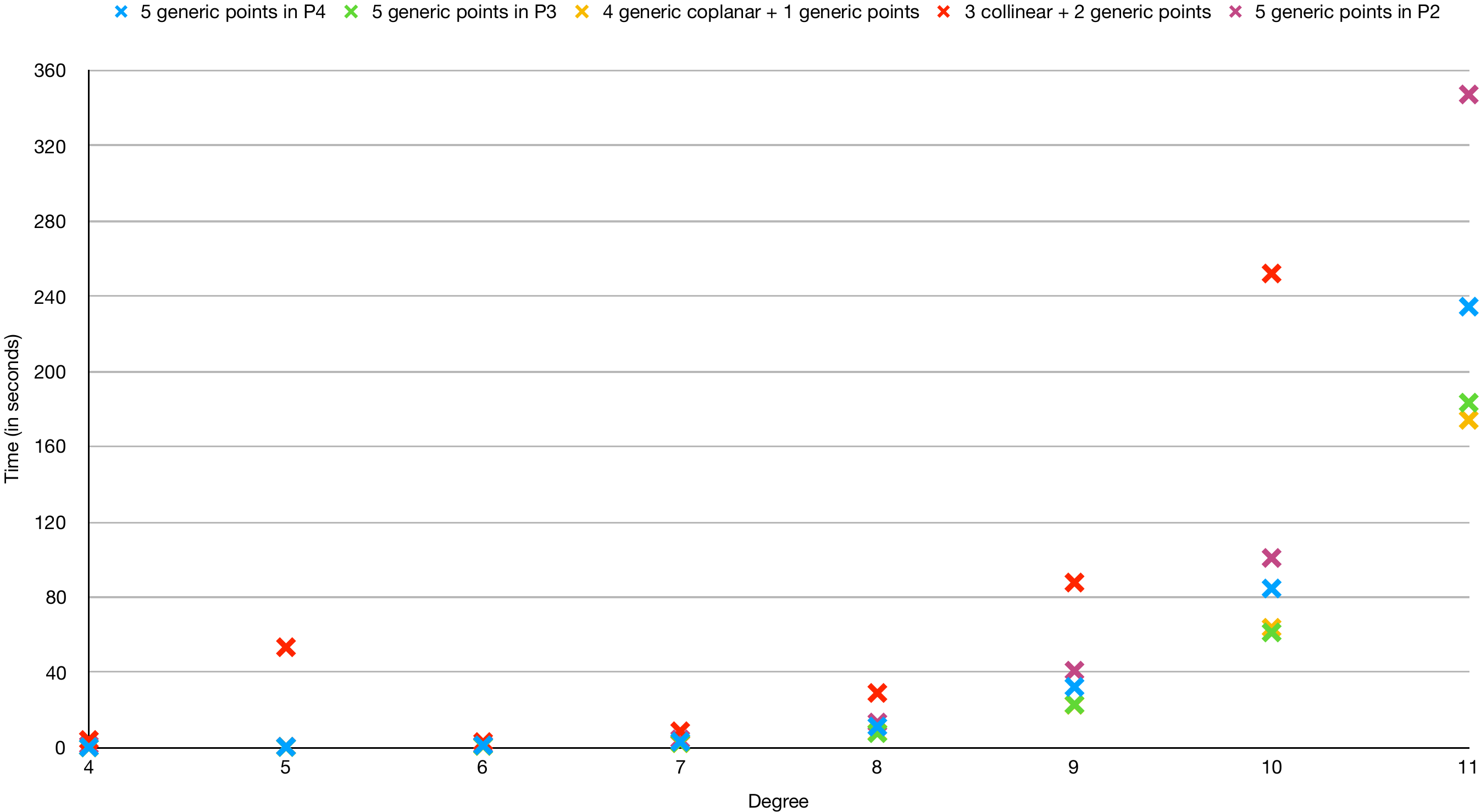}}
\caption{Tests with fixed number of variables equal to $6$. Note that,
  in the case of three collinear and two generic points, the degree
  $5$ is special because it requires a procedure different than the cases with larger degree; it follows from Theorem \ref{thm: main}(12).}
\label{fig: vars6}
\end{figure}

\bibliographystyle{alpha}
\bibliography{Oneto_references.bib,BM_references.bib}
\end{document}